\documentclass[leqno,a4paper]{amsart}

\renewcommand{\baselinestretch}{1.5}
\usepackage{amsmath,amsfonts,amsthm,amssymb,dsfont}

\usepackage{times}
\usepackage{microtype}

\usepackage{cite}

\usepackage{mathrsfs}
\usepackage{enumerate, xspace}
\usepackage{graphicx}
\usepackage{color}
\usepackage[makeroom]{cancel} 

    \usepackage[all, knot]{xy}
    \xyoption{arc} 

\usepackage{pgf,tikz}\usepackage{mathrsfs}\usetikzlibrary{arrows}

\usepackage{verbatim}

\usepackage{caption}
\usepackage{subcaption}
\usepackage[pdfauthor={Arman Darbinyan and Markus Steenbock},pdftitle={Embeddings into left-ordered simple groups},pdfkeywords={embedding theorems, left-ordered groups, simple groups, word problem, computable algebra},pdftex]{hyperref}

\hypersetup{colorlinks,%
citecolor=black,%
filecolor=black,%
linkcolor=black,%
urlcolor=black,%
}

\theoremstyle{plain}		
				\newtheorem{them}{Theorem}

				\newtheorem{thm}{Theorem}[section]
				\newtheorem{prop}[thm]{Proposition}
				\newtheorem{lem}[thm]{Lemma}		
						\newtheorem{cor}[thm]{Corollary}

\theoremstyle{definition}		\newtheorem{df}[thm]{Definition}
						\newtheorem{ex}[thm]{Example}	
						\newtheorem{step}{Step}	
				
\newtheorem{rem}[thm]{Remark}

	\theoremstyle{remark} 	

\DeclareMathOperator{\T}{T}
\DeclareMathOperator{\h}{H}
\DeclareMathOperator{\p}{P}
\DeclareMathOperator{\R}{RiSt}
\DeclareMathOperator{\s}{Sp}
\DeclareMathOperator{\ft}{Ft}
\DeclareMathOperator{\coWP}{coWP}
\DeclareMathOperator{\WP}{WP}
\DeclareMathOperator{\supp}{supp}
\DeclareMathOperator{\B}{bij}
\DeclareMathOperator{\Bp}{B^{+}}
\DeclareMathOperator{\homeo}{Homeo}

\title{Embeddings into left-orderable simple groups}
\date{\today}
\author{Arman Darbinyan}
\address{Texas A\&M, USA}
\email{arman.darbin@gmail.com}
\author{Markus Steenbock }
\address{IRMAR\\ Univ Rennes et CNRS\\ 35000 Rennes \\ France}
\email{markus.steenbock@univ-rennes1.fr}


\begin{document}

\begin{abstract}
We prove that every countable left-ordered group embeds into a finitely generated left-ordered simple group. Moreover, if the first group has a computable left-order, then the simple group also has a computable left-order. 
 We also obtain a Boone-Higman-Thompson type theorem for left-orderable groups with recursively enumerable positive cones. These embeddings are Frattini embeddings, and isometric whenever the initial group is finitely generated. 
 
 Finally, we reprove Thompson's theorem on word problem preserving embeddings into finitely generated simple groups and observe that the embedding is isometric. 
\end{abstract}

 \maketitle

 \section{Introduction}

A group is \emph{simple} if it has no proper non-trivial normal subgroups. Infinite finitely generated simple groups were  discovered in \cite{higman_finitely_1951}. In fact, every  countable group embeds into a finitely generated simple group \cite{hall_embedding_1974,gorjuskin_imbedding_1974}, see also \cite{schupp_embeddings_1976,thompson_word_1980}.

\subsection{Left-order preserving embeddings into simple groups}

A group is left-ordered if it has a  linear order that is invariant under multiplications from the left.

By \cite[Theorem 4.5]{kim_chain_2019}, every finitely generated left-ordered groups embeds into a finitely generated left-ordered group whose derived subgroup is simple.

 Infinite finitely generated simple and left-ordered groups were discovered by Hyde and Lodha in \cite{hyde_finitely_2018}, see also \cite{mattebon_groups_2018,hyde_uniformly_2019}. 
We extend the construction of such groups \cite{hyde_finitely_2018,mattebon_groups_2018} as follows.

\begin{them}\label{main1} Every countable left-ordered group $G$ embeds into a finitely generated left-ordered simple group $H$. Moreover, the order on $H$ continues the order on $G$.
\end{them}

We also study additional geometric and computability properties of such embeddings, see Remark \ref{IR: Frattini} and Theorem \ref{main2}.

A subgroup $G$ of $H$ is called \emph{Frattini embedded} if any two elements of $G$ that are conjugate in $H$ are also conjugate in $G$.  Also, if there exist finite generating sets $X$ and $Y$ of $G$ and $H$, respectively, such that the word metric of $G$ with with respect to $X$ coincides with the word metric of $G$ with respect to $Y$, then it is said that $G$ is \emph{isometrically embedded} in $H$.

\begin{rem}\label{IR: Frattini} The embedding of Theorem \ref{main1} can be chosen to be a Frattini embedding. If $G$ is finitely generated, the embedding is also isometric. 
\end{rem}

 A systematic study of computability aspects of orders on groups was initiated in \cite{downey_recursion_1986}, see also \cite{downey_survey}. 
A left-order is \emph{computable} if it is decidable whether a given element is positive, negative or equal to the identity. In particular, a finitely generated computably left-ordered group has a decidable word problem. 

The following theorem is the computable version of Theorem \ref{main1}.
\begin{them}\label{main2} Every countable computably left-ordered group $G$ embeds into a finitely generated computably left-ordered simple group $H$. Moreover, the order on $H$ continues the order on $G$.

In addition, the embedding is a Frattini embedding, and if $G$ is finitely generated, then it  is isometric. 
\end{them}

\subsection{Boone-Higman and Thompson's theorem revisited}

A landmark result on computability in groups is the Boone-Higman theorem. It states that a finitely generated group has decidable word problem if and only if it embeds into a simple subgroup of a finitely presented group. Thompson strengthened Boone-Higman's theorem by showing that the simple group can be chosen to be finitely generated \cite{thompson_word_1980}.  The next theorem is a version of Thompson's theorem that, in addition, preserves the geometry of the group.

\begin{them}[cf. Theorem \ref{T: thompson}]\label{T: main 4}  Every 
 countable group $G$ embeds into a finitely generated simple group $H$ such that if $G$ has decidable word problem, then so does $H$. 
 
In addition, the embedding is a Frattini embedding. If $G$ is finitely generated, then the embedding is isometric. 
\end{them}

\begin{rem} Belk and Zaremsky \cite[Theorem C]{belk_zaremsky} recently proved that every finitely generated  group isometrically embeds into a finitely generated simple group, but they did not study the Frattini property or computability properties of their embedding. Their result and Theorem \ref{T: main 4} strengthen a theorem of Bridson, who proved that every finitely generated group quasi-isometrically embeds into a finitely generated group without any non-trivial finite quotient  \cite{bridson_embedding_98}. 
\end{rem}

\begin{rem} If the group $G$ in Theorem \ref{T: main 4}  is not finitely generated, instead of saying $G$ has decidable word problem, it is more common to say that $G$ is a {computable} group (see Definition \ref{D: computable group} below).
 \end{rem}

Bludov and Glass obtained a left-orderable version of the Boone-Higman theorem by showing that a left-orderable group has decidable word problem if and only if it embeds into a simple subgroup of a finitely presented left-orderable group \cite[Theorem E]{bludov_word_2009}. In this context, it is natural to ask whether the simple group can be made finitely generated, cf. \cite[p. 251, Problem 4]{glass_ordered_1981}. 

The next theorem answers this question in the positive given that the set of positive elements is recursively enumerable. Namely, the following theorem holds.

\begin{them}
\label{theorem-bludov-glass}
Let $G$ be a left-orderable finitely generated group that has a recursively enumerable positive cone with respect to some left-order. Then $G$ has decidable word problem if and only if $G$ embeds into a finitely generated simple subgroup of a finitely presented left-orderable group.
\end{them}

\begin{rem}
The existence of left-orderable groups with decidable word problem that do not embed in a group with computable left order was shown in \cite{arman_new}.

Also, the existence of finitely generated left-orderable groups with decidable word problem but without recursively enumerable positive cone is first  shown in \cite{arman_new}. Earlier, the analogous result for countable but not finitely generated groups was shown in \cite{harrison_left-orderable_2018}.
\end{rem}

The question whether Theorem \ref{theorem-bludov-glass} holds without the assumption that $G$ has a left-order with recursively enumerable positive cone remains open. Also it is open whether a finitely generated left-orderable simple group with decidable word problem but without recursively enumerable positive cone exists.

\subsection{Sketch of the embedding constructions}

We sketch the proof of Theorems \ref{main1} and \ref{main2}. We start with a countable computably left-ordered group $G$.
\begin{step}[Embedding into a finitely generated group]
 By a classical wreath product construction \cite{neumann_embedding_1960} every countable  left-ordered group embeds into a $2$--generated left-ordered group. A version of this embedding construction with additional computability properties was established in \cite{darbinyan_group_2015}. We use the construction from \cite{darbinyan_group_2015} (see Theorem~\ref{L: embedding into 2 generated}) to embed the initial left-orderable countable group $G$ into a two-generated left-orderable group that also preserves the computability properties of the left-order on $G$. 
\end{step}

\begin{step}[Embedding into a perfect group] A group is perfect if it coincides with its first derived subgroup. By Step 1 we assume that $G$ is finitely generated.  We let $T(\varphi)$ be a  finitely generated left-ordered simple group  of \cite{mattebon_groups_2018}. We note that $T(\varphi)$ is computably left-ordered and $G$ embeds into a finitely-generated left-orderable perfect subgroup $G_1$ of $G\wr_{\mathbb{R}} T(\varphi)$ that preserves the computability property of the left-order on $G$, see Theorem \ref{T: splinter}. Our construction might be considered as a modification of a similar embedding result from \cite{thompson_word_1980}.
\end{step}

\begin{step}[Embedding into a simple group of piecewise homeomorphisms of flows] Finally, let  $G_1$ be a finitely generated (computably) left-ordered perfect group in which $G$ embeds. 
  We embed $G_1$ into a finitely generated (computably) left-ordered simple group. To this end, we extend the construction of \cite{mattebon_groups_2018}. In \cite{mattebon_groups_2018}, Matte-Bon and Triestino construct a finitely generated left-orderable simple group $T(\varphi)$ of piecewise linear homeomorphisms of flows of the suspension of a minimal subshift $\varphi$, see Subsection \ref{S: T(Phi)}.
  
  The main observation is that every group $H$ of piecewise homeomorphisms of an interval with countably many breakpoints (see Definition \ref{Def: the new space of functions}) embeds into a subgroup $T(H,\varphi)$ of piecewise homeomorphisms of flows of the suspension, see Definition \ref{D: T(G,Phi)}. We then study the subgroup $T(H,\varphi)$. In particular, it is finitely generated if $H$ is so. Just as in \cite{mattebon_groups_2018}, a standard commutator argument 
   implies that it is simple given that $H$ is perfect, and if $H$ preserves the orientation of the interval, then it is also left-orderable.
  
 Finally, we use the dynamical realisation of left-orderability: every left-ordered group embeds into the group of orientation preserving homeomorphisms of an interval. We use this embedding to conclude that $G_1$, and hence also $G$, embeds into the finitely generated left-ordered simple group $T(G_1,\varphi)$. To analyze the required computability aspects as well as to show that the embeddings are isometric and Frattini, we use a modified version of the dynamical realization of left-orderability, see Proposition \ref{prop-modified-dyn-rel}. 
\end{step}

If $G$ has decidable word problem, it embeds into a group of computable piecewise homeomorphisms of an interval \cite[\S 3]{thompson_word_1980}. If we use this embedding in Step 3 of the above construction, then we obtain the aforementioned result of \cite{thompson_word_1980}, Theorem \ref{T: main 4}.

\subsection{Plan of the paper}

In Section \ref{S: computability}, we review computable groups and computably left-ordered groups. In particular, we explain the computability of the standard dynamical realization of left-orderabitity. 

After that, we come to the main parts of our paper. In Section \ref{S: homeo of flows}, we discuss Step 3, that is, we extend Matte-Bon and Triestino's construction of left-orderable finitely generated simple groups in order to embed perfect groups into finitely generated simple groups. Step 2, our version of Thompson's splinter group construction, is discussed in Section \ref{S: embedding into perfect}. Step 1 is reviewed in Section \ref{S: embedding into 2 generated}.

Finally, we prove Theorems \ref{main1}, \ref{main2} and \ref{theorem-bludov-glass}. To analyze the computability aspects required by Theorem \ref{main2} as well as to obtain the isometry and Frattini properties of the embeddings, we introduce a stronger version of the standard dynamical realization of left-orderability that we call modified dynamical realization(see Section \ref{S: modified}).
In Section \ref{S: Thompson's theorem}, we prove Theorem \ref{T: main 4} using the groups of piecewise homeomorphisms of flows discussed in Section \ref{S: homeo of flows}.

%
\medskip
 \noindent
\textbf{Acknowledgements.} 
We thank Y. Lodha, M. Triestino and M. Zaremsky for their interest and useful comments on a previous version of this work. 
The first named author thanks Universit\'e Rennes-I for hospitality and financial support and was supported by ERC-grant GroIsRan no.725773 of A. Erschler. The second named author was supported by ERC-grant GroIsRan no.725773 of A. Erschler and by Austrian Science Fund (FWF) project J 4270-N35.

 \section{Computability on groups} \label{S: computability}

We collect some facts from computability theory on groups, cf. \cite{frohlich_effective_1956,rabin_computable_1960,matsev_constructive_1961,downey_recursion_1986}.

A function $f:\mathbb{N}\to \mathbb{N}$ is \emph{computable} if there is a Turing machine such that it outputs the value of $f$ on the input. A subset of $\mathbb{N}$ is \emph{recursively enumerable} if there is a computable map (i.e. enumeration) from $\mathbb{N}$ onto that set. Moreover, it is \emph{recursive} if, in addition, its complement is recursively enumerable as well. 

 Similarly, a function $f:\mathbb{Q}\to \mathbb{Q}$ is \emph{computable} if there is a Turing machine that, for every input $(m,n)\in \mathbb{N}\times \mathbb{N}$,  outputs $(p,q)\in \mathbb{N}\times \mathbb{N}$ such that $f\left(\frac{m}{n}\right)=\frac{p}{q}$.
 
Moreover, if $J$ is an interval in $\mathbb{R}$, then we call a function $f:J\to \mathbb{R}$ \emph{computable} if its restriction to the rational numbers in $J$ maps to $\mathbb{Q}$ and this restriction is computable.

\subsection{Group presentations and the word problem}
Let $S$ be a finite set. We denote by $(S\cup S^{-1})^*$ the set of all finite words over the alphabet $S\cup S^{-1}$.

\begin{df}[word problem]
Let $G=\langle S \rangle$ be a finitely generated group. The word problem is decidable if the set $\WP(S):=\{w \in (S \cup S^{-1})^* \mid w=_G 1 \}$ is recursive.
\end{df}
The decidability of the word problem does not depend on the choice of the finite generating set.

\subsection{Computable groups}

For a countable group $G=\{g_1,g_2,\ldots\}$, let $ \mathfrak{m}:\mathbb{N}\times \mathbb{N} \to \mathbb{N} $   
be the function such that 
\begin{align*}
\hbox{$\mathfrak{m}(i,j)=k$ if $g_ig_j=g_k$.}
\end{align*}

\begin{df}\label{D: computable group}
A countable group $G$ is \emph{computable} if there exists an enumeration of its elements $G=\{g_1,g_2,\ldots\}$ such that the corresponding $ \mathfrak{m}:\mathbb{N}\times \mathbb{N} \to \mathbb{N} $ is computable.
\end{df}

\begin{rem}\label{R: computable group and word problem}
 A finitely generated group is computable if and only if it has decidable word problem.
\end{rem}

\subsection{Computably left-ordered groups}

An order $\preceq$ on $\mathbb{N} \times \mathbb{N}$ is \emph{computable} if there is a Turing machine that takes a pair $(i,j)\in \mathbb{N}\times \mathbb{N}$ as input and decides whether or not $i \preceq j$. 

For a countable linearly ordered enumerated set $S=\{s_1,s_2,\ldots\}$, let $ \preceq$ on $\mathbb{N}\times \mathbb{N}$  
be the relation such that 
\begin{align*}
\hbox{$i\prec j$ if $s_i$ is smaller than $s_j$ and $i=j$ if $s_i$ is equal to $s_j$.}
\end{align*}

A countable set $S$ is \emph{computably orderable} with respect to the enumeration $S=\{s_1,s_2,\ldots\}$ if there is a linear order on $S$ such that the corresponding to it order relation $ \preceq$ on $\mathbb{N}\times \mathbb{N}$ is computable.  

\begin{df}\label{D: computably left ordered group}
A countable group $G$ is \emph{computably left-orderable} with respect to the enumeration $G=\{g_1,g_2,\ldots \}$ if  there is a left-order $\preceq$ on $G$ such that the corresponding order relation $ \preceq$ on $\mathbb{N}\times \mathbb{N}$ is computable. In this case $\preceq$ is called computable left-order on $G$ with respect to the enumeration  $G=\{g_1,g_2,\ldots \}$. \end{df}
\begin{rem}
In case $G=\langle S \rangle$, $|S|<\infty$, $G$ is computably left-orderable with respect to some enumeration if and only if there is a left-order $\preceq$ on $G$ such that the set $\{ w \in (S \cup S^{-1})^* \mid 1 \preceq w \} \subseteq (S \cup S^{-1})^*$ is a recursive set. In this case $\preceq$ is called computable left-order on $G$, and its computability property does not depend on the choice of the finite generating set, see \cite{arman_new} for details.
\end{rem}

\begin{rem} \label{R: computable group and computable order}
Every computably left-orderable group is computable. In particular, every finitely generated computably left-ordered group has decidable word problem. 
\end{rem}

By \cite{harrison_left-orderable_2018} there is a left-orderable computable group without any computable left-order. In fact, there is a finitely generated  orderable computable group without any computable order \cite{arman_new}.

\begin{ex} The natural order on the group of rational numbers is computable.  
\end{ex}

\begin{ex}[Thompson's group \texorpdfstring{$F$}{F}] \label{E: Thompson's group}

A \emph{dyadic point} in $\mathbb{R}$ is one of the form $\frac{n}{2^m}$, for some $n\, ,m\in \mathbb{Z}$. An interval is \emph{dyadic} if its endpoints are dyadic. 

Let  $J$ be a closed dyadic interval in $\mathbb{R}$.   We denote by $\mathbb{Q}_J$ the set of the rational points on $J$.  We denote by $F_J$ the group of piecewise linear homeomorphisms of $J$ that are differentiable except at finitely many dyadic points and such that the respective derivatives, where they exist, are powers of $2$.
 
 The group $F_J$ is isomorphic to Thompson's group $F$,  see e.g. \cite[\S 1]{cannon_introductory_1996}. Therefore, it is $2$--generated and left-orderable, see e.g. \cite[Corollary 2.6, Theorem 4.11]{cannon_introductory_1996}. Moreover, the word problem in $F$ is decidable, cf. \cite{thompson_word_1980}.
 
 We define the left-order on $F_J$ in the following way, cf. \cite[Theorem 4.11]{cannon_introductory_1996}: let $\mathbb{Q}_J=\{q_1,q_2,\ldots\}$ be a fixed recursive enumeration. Let $f,g\in F_J$ be distinct and let $i_0$ be the minimal index such that $f(q_{i_0})\not= g(q_{i_0})$. Then $f<g$ if $f(q_{i_0})< g(q_{i_0})$. 
 
 In fact, this order is computable: indeed, let $f,g\in F_J$ be given as words in a finite generating set. As the word problem in $F_J$ is decidable, the case of $f=g$ can be computably verified.  We note that the elements of $F_J$ are computable functions. In addition, an element of $F_J$ is uniquely determined by its restriction to the rationals.  
 Thus, if $f\not=g$, the minimal index $i_0$ such that $f(i_0)\not= g(i_0)$ exists and can be computably determined. Therefore, the order is computable.
  \end{ex}

\subsection{Positive cones}

If $G$ is left-ordered, then the \emph{positive cone} is the set of all positive elements of $G$. We note that the positive cone is a semigroup. 
In fact, if $G$ admits a linear order such that the positive elements generate a semigroup in $G$, then the linear order is a left-order on $G$, see \cite{deroin_groups_2014,clay_ordered_2016}. 
  
  \begin{lem}
  \label{lem: computable left-orders}
  Let $G=\{g_1, g_2, \ldots \}$ be a finitely generated group with a fixed enumeration, and `$\prec$' be a left-order on $G$. Then `$\prec$' is computable  if and only if its positive cone is recursively enumerable and the word problem in $G$ is decidable.
  \end{lem}
  \begin{proof}
  	If the order is computable, then the word problem is decidable, see Remark \ref{R: computable group and computable order}. In addition, there is a partial algorithm to confirm that a positive element, given as a word in the generators of $G$, is positive. This implies that the positive cone is recursively enumerable. 
  	
 On the other hand, if the positive cone is recursively enumerable and the word problem decidable, let $w$ be a word in the generators of $G$. We first computably determine whether or not $w=1$. If $w=1$, we stop. Otherwise either $w$ or $w^{-1}$ is in the positive cone of $G$. As the positive cone is recursively enumerable, there is a partial algorithm to confirm that a positive element is in the positive cone. We simultaneously run this algorithm for $w$ and $w^{-1}$. As one of these elements is positive, it stops for $w$ or $w^{-1}$. We thus know whether $w$ is positive or negative. This completes the proof. 
  \end{proof}

\subsection{Dense orders}

A linear order $\preceq$ on a set $S$ is \emph{dense} if for any $g \prec h \in S$, there exists $g' \in S$ such that $g \prec g' \prec h$.

  Recall that by $\mathbb{Q}_J$ the set of the rational points on an interval $J\subset \mathbb{R}$. We fix a recursive enumeration $\mathbb{Q}_J=\{q_0, q_1, \ldots \}$ such that the natural order on $\mathbb{Q}_J$ is computable with respect to this enumeration.
  
\begin{lem}[cf. Theorem 2.22 of \cite{clay_ordered_2016}]
\label{Lem: Cantor's argument} Let $J\subset \mathbb{R}$ be an interval. 
Let $S=\{s_0, s_1, \ldots\}$ be a countable ordered set. If the order on $S$ is dense and does not have maximal and minimal elements, then there is an order preserving bijection $\Phi: S \to \mathbb{Q}_J$. 

If, in addition, the order on $S$ is computable, then the map $i \mapsto \Phi(s_i)$ is computable. 
\end{lem}

We recall the proof of this lemma, that we will later modify to prove Lemma \ref{Lem: Cantor's argument modified}.
\begin{proof}
We define $\Phi:  S \to \mathbb{Q}_J$ iteratively as $\Phi: s_{j_i} \mapsto q_{j_i}$ for $i \in \mathbb{N}$. First, define $s_{j_0} = s_0$ and $q_{j_0} = q_0$. Now assuming that $S_k:=\{s_{j_0}, \ldots ,s_{j_k}\}$ and $Q_k:=\{q_{j_0}, \ldots ,q_{j_k} \}$ are already defined, let us define its extension according to the following procedure:
\begin{itemize}
\item[(1)] Choose the smallest $i$ such that $s_i \notin S_k$ and set $S_{k+1}=S_k\cup \{s_i\}$. Choose the smallest $j$ such that $q_j \notin Q_k $ and $\Phi: S_k \cup \{s_i \} \to Q_k \cup \{q_j \}$ is an order preserving bijection. Set $s_{j_{k+1}}= s_{i}$ and $q_{j_{k+1}}= q_{j}$.
\item[(2)] Choose the smallest $j$ such that $r_j \notin Q_{k+1}$, and choose the smallest $i$ such that $s_i \notin S_{k+1}$ and $\Phi^{-1}: Q_{k+1} \cup \{q_j\} \to S_{k+1}\cup \{s_i\}$ is an order preserving bijection. Set $s_{j_{k+2}}= s_{i}$ and $q_{j_{k+2}}= q_{j}$.
\item[(3)] Repeat the process starting from Step 1.
\end{itemize}
Since the orderings of $S$ and $\mathbb{Q}_J$ are computable with respect to the fixed  enumerations, the above described iterative procedure of defining $\Phi$ is also computable. Therefore, the map $i \mapsto \Phi(s_i)\in \mathbb{Q}_J$ is computable.
\end{proof}

\begin{rem}\label{R: embedding in dense} If $G$ is left-ordered, then the lexicographical left-order on the group $G\times \mathbb{Q}$ is dense (and has no minimal or maximal elements). In addition, if $G=\{g_1, g_2, \ldots \}$ has a computable left-order, the lexicographical left-order on $G\times \mathbb{Q}$ is computable with respect to the induced enumeration. Moreover, the standard embedding $G \hookrightarrow G\times \mathbb{Q}$ that sends $g\mapsto (g,0)$ is computable and a Frattini embedding.
 \end{rem}

\subsection{Dynamical realization of computably left-ordered groups}\label{S: dynamical}

Let $J$ be an interval in $\mathbb{R}$. We denote the group of homeomorphisms of $J$  by $\homeo(J)$, and the subgroup of orientation preserving homeomorphisms of $J$ by $\homeo^+(J)$. 

We note that for every interval $J\subset \mathbb{R}$, every countable left-ordered group $G$ admits an embedding of $G$ into $\homeo^+(J)$, see e.g. \cite[\S 2.4]{clay_ordered_2016} \cite[Proposition 1.1.8]{deroin_groups_2014}. We also note the following fact.  
  

\begin{prop}\label{P: computable left order} Let $G$ be a countable group.

 If $G$ is left-orderable, then there is an embedding $\rho_G:G \hookrightarrow \homeo^+(J)$  such that, for all $g\in G \setminus \{1\}$, the map $\rho_G(g): J \rightarrow J$ does not fix any rational interior point of $J$. 

 If $G$ is computably left-orderable, then, in addition, all the maps $\rho_G(g)$ can be granted to be computable. 
\end{prop}

We actually need a strong variant of Proposition \ref{P: computable left order}, see Proposition \ref{prop-modified-dyn-rel}, but to the best of our knowledge, the computability aspect of Proposition  \ref{P: computable left order} does not exist in the literature neither.
For this reason we decided to include a proof of Proposition \ref{P: computable left order}. We analyze computability aspects based on the proof given in  \cite[\S 2.4]{clay_ordered_2016}.

By Remark \ref{R: embedding in dense}, we may assume that the order on $G$ is dense. Then, by Lemma \ref{Lem: Cantor's argument}, there is an order preserving bijection $\Psi: G \to \mathbb{Q}_J$.

\begin{df} 
\label{def-dyn-rel-embedding}
Let $\Psi: G \to \mathbb{Q}_J$ be an order preserving bijection. We define $\rho_G^{\Psi}:G \to \homeo^+(J)$
by prescribing 
   $\rho_G^{\Psi}(g_i)(\Psi(h))= \Psi(g_ih)$ on the dense subset $\Psi(G)=\mathbb{Q}_J\subset J$. 
\end{df}

We note the following. 
\begin{lem}\label{L: embedding computable left order} Let $\Psi: G \to \mathbb{Q}_J$ be an order preserving bijection. The map $\rho_G^{\Psi}:G \to \homeo^+(J)$ is an embedding. 
Moreover, if the map $i \mapsto \Psi(s_i)$ is computable, then for all $i \in \mathbb{N}$, $\rho_G^{\Psi}(g_i)$ is computable. \qed
\end{lem}

\begin{lem}\label{L: dynamical realisation no fixed points} Let $\Psi: G \to \mathbb{Q}_J$ be an order preserving bijection. If $x\in \mathbb{Q}_J$ such that $\rho_G^{\Psi}(g)(x)= x$, then $g=1$. 
\end{lem}
\begin{proof}
	Indeed, suppose that $\rho_G^{\Psi}(g)(x) = x$. Let $h = \Psi^{-1}(x)$. Then by definition of $\rho_G^{\Psi}$, we have $\rho_G^{\Psi}(x) = \Psi(gh)$, i.e. $\Psi(h)=\Psi(gh)$. But since $\Psi: G \rightarrow \mathbb{Q}_J$ is a bijection, $h=gh$ and $g=1$.
\end{proof}

\begin{proof}[Proof of Proposition \ref{P: computable left order}] Suppose $G=\{g_1, g_2, \ldots \}$ has a computable left-order with respect to the given enumeration. By Lemma \ref{Lem: Cantor's argument}, we may assume that the map $i \mapsto \Psi(g_i)$ is computable. 
By Lemmas \ref{L: embedding computable left order} and \ref{L: dynamical realisation no fixed points}, $\rho_G^{\Psi}:G \to \homeo^+(J)$ satisfies the properties required by Proposition \ref{P: computable left order}.
\end{proof}


 \section{Groups of piecewise homeomorphisms of flows} \label{S: homeo of flows}

We first collect definitions and facts on groups of piecewise linear homeomorphisms of flows from \cite{mattebon_groups_2018}.  As every countable group embeds as a subgroup in a group of piecewise  homeomorphisms of flows, we then start to study such groups in more generality.

We recall from Example \ref{E: Thompson's group} that a \emph{dyadic point} in $\mathbb{R}$ is one of the form $\frac{n}{2^m}$ for some $n\, ,m\in \mathbb{Z}$. Moreover, for a dyadic interval $J$, $F_J$ is Thompson's group acting on $J$.
 
 \subsection{Minimal subshifts} \label{S: cantor systems}

Let $A$ be a finite alphabet and $\varphi$ a shift on $A^{\mathbb{Z}}$. If $X$ is a closed and shift-invariant subset of $A^{\mathbb{Z}}$, then $(X,\varphi)$ is a dynamical system that is called \emph{subshift}. A subshift is \emph{minimal} if the set of $\varphi$--orbits is dense in $X$.


Let $(X,\varphi)$ be a minimal subshift of $A^{\mathbb{Z}}$. Then $X$ is totally disconnected and Hausdorff, and every $\varphi$--orbit is dense in $X$.

The \emph{suspension} (or mapping torus) $\Sigma$ of $(X,\varphi)$ is the quotient of $X \times \mathbb{R}$ by the equivalence relation defined by 
$
(x,t) \sim (\varphi^n (x), t-n),\, n\in \mathbb{Z}. 
$
We denote the corresponding equivalence class of $(x, t) \in X\times \mathbb{R}$  by $[x,t]$.

The map $\Phi_t$ that sends $[x,s]$ to $[x,s+t]$ is a homeomorphism and defines a flow $\Phi$ on $\Sigma$, the \emph{suspension flow}, so that $(\Sigma,\Phi)$ is a dynamical system as well. The orbits of the suspension flow are homeomorphic to the real line. 

We denote by $\h(\varphi)$ the group of homeomorphisms of $\Sigma$ that preserves the orbits of the suspension flow, and by $\h_0(\varphi)$ the subgroup of $\h(\varphi)$ that, in addition,  preserves the orientation on each orbit.

\subsection{The group  \texorpdfstring{$\T(\varphi)$}{of piecewise dyadic homeomorphisms of flows}} \label{S: T(Phi)}

Let $C$ be a clopen subset of $X$ and let $J\subset \mathbb{R}$ be of diameter $< 1$. The embedding of $C\times J$ into $X\times \mathbb{R}$ descends to an embedding into $\Sigma$ that we denote by $\pi_{C,J}$.

For every  clopen $C\subset X$ and subset $J$ of diameter $< 1$ in $\mathbb{R}$, the map $\pi_{C,J}$ is a \emph{chart} for the suspension, whose image is denoted by $U_{C,J}$. 
 If $z$ is in the interior of $U_{C,J}$, then $\pi_{C,J}$ is a \emph{chart at $z$}.

 \begin{df}[Dyadic chart] \label{D: dyadic chart} Let $C$ be a clopen subset of $X$, and let $J$ be a dyadic interval of length $< 1 $ in $\mathbb{R}$. Then $\pi_{C,J} : C\times J \hookrightarrow \Sigma$ is called \emph{dyadic chart}. 
\end{df}
 
\begin{df}[Dyadic map] A \emph{dyadic map} is a map $f$ of real numbers such that $f(x)=\lambda x +c$, where $\lambda $ is  a power of $2$ and $c$ is a dyadic rational.
\end{df} 
 \begin{df}[Definition 3.1 of {\cite{mattebon_groups_2018}}] \label{D: T(Phi)} The group $\T(\varphi)$ is the subgroup of $\h_0(\varphi)$ consisting of all elements $h\in \h_0(\varphi)$ such that for all $z\in \Sigma$ there is a dyadic chart $\pi_{C,J}$ at $z$ and a piecewise dyadic map $f:J\to f(J)$ with finitely many breakpoints such that the restriction of $h$ to $U_{C,J}$ is given by  
$
[x,t]\mapsto [x,f(t)]
$.
\end{df}

We recall that $F_J$ denotes the group of piecewise dyadic homeomorphisms of $J$ with finitely many breakpoints. 
\begin{df}\label{D: Elements of T(Phi)} Let $\pi_{C,J}$ be a dyadic chart and let $f\in F_J$. Then $f_{C,J}$ is the map in $\T(\varphi)$ whose restriction to $U_{C,J}$ is given by 
\begin{align*}
[x,t] \mapsto [x,f(t)]
\end{align*}
 and that is the identity map elsewhere. We let $F_{C,J}$ be the subgroup of $\T(\varphi)$ generated by the elements $f_{C,J}$ for all $f\in F_J$.
\end{df}

The group $\T(\varphi)$ is infinite, simple, left-ordered and finitely generated \cite[Corollary C]{mattebon_groups_2018}. As noted in \cite{mattebon_groups_2018}, the first examples of such groups \cite{hyde_finitely_2018} are subgroups of $\T(\varphi)$.

In Section \ref{S: normal subgroups}, we revisit the proof of simplicity given in \cite{mattebon_groups_2018}. In Section \ref{S: left-orders}, we revisit the proof of left-orderability given in \cite{mattebon_groups_2018}. To this end we note the following. 

\begin{lem}[Lemma 3.4 of \cite{mattebon_groups_2018}]\label{L: orbits} For every $z\in \Sigma$, the $T(\varphi)$--orbit of $z$ is dense in the $\Phi$--orbit of $z$. In particular, the $\T(\varphi)$--action on $\Sigma$ is minimal.
\qed
\end{lem}

For any group $H$, we denote by $H'$ the first derived subgroup of $H$.

\begin{lem}[Lemma 4.8 of \cite{mattebon_groups_2018}]\label{L: coverings} Let $C\subset X$ be clopen and $J\subset \mathbb{R}$ be dyadic. If $C\times J$ is covered by a family $\{C_i\times J_i\}_{i \in \mathcal{I}}$ for clopen $C_i\subset C$ and dyadic intervals $J_i\subset J$, then $F'_{C,J}$ is contained in the group generated by $\bigcup _{i \in \mathcal{I}}  F'_{C_i,J_i}$
\qed
\end{lem}

We assume without restriction that $(X,\varphi)$ is a minimal subshift over the two letter alphabet $A=\{0, 1\}$. 
 For $k,\, n \in \mathbb{Z}$ and a word $w=a_{0}a_{1}\ldots a_{k}$ over $A$, we denote by $C_{n,w}$ the cylinder subset of $X$ consisting of sequences $(x_i)_{i\in \mathbb{Z}}$ such that $x_{n}x_{n+1}\ldots x_{n+k} = w$. 
 
As a matter of fact, the cylinder subsets are clopen and form a basis for the topology of $X$.  
 We note that 
 \begin{align*}
 \varphi \left(C_{n,w}\right) = C_{{n-1},w}. 
\end{align*}  
Let $I_0:=[-1/4,1/2]$ and $I_1:=[1/4,9/8]$. 
\begin{lem}[Proposition 6.2 of \cite{mattebon_groups_2018}]\label{L: finite generation  of T(Phi)} The group $\T(\varphi)$ is generated by  $F_{X,I_0}$, $F_{C_{0,0},I_1}$ and $F_{C_{0,1},I_1}$. In particular, $\T(\varphi)$ can be generated by six elements.\qed
\end{lem}


\subsection{Piecewise homeomorphisms and group embeddings}\label{S: piecewise homeo of flows}

\begin{df}
\label{Def: the new space of functions} A bijection $h:I\to J$ of subsets $I,J\subseteq \mathbb{R} $ is a \emph{piecewise homeomorphism} if
\begin{itemize}
\item  there are half-open pairwise disjoint intervals $I_i=[x_i,y_i)$, $i=1, 2, \ldots,$ whose union is $I$,
\item for all of these $I_i$ the restriction of $h$ to $I_i$ is a homeomorphism onto its image,
\end{itemize}
If, in addition, the intervals $I_i$ and $h(I_i)$ are dyadic, we say that $h$ has \emph{dyadic breakpoints}.
If, the restrictions of $h$ to the intervals $I_i$ are dyadic maps, we say that $h$ has \emph{dyadic pieces}.
\end{df}

If $S$ is a set, $\B(S)$ denotes the group of permutations of $S$.   

 Let us fix a half-open interval $\mathfrak{J}=[x, y)$ that is strictly contained in $[0,1]$.   
Let $\mathcal{C}(\mathfrak{J}) \subset \B(\mathfrak{J})$ denote the subgroup of all piecewise homeomorphisms with dyadic breakpoints on $\mathfrak{J}$. The subgroup of $\mathcal{C}(\mathfrak{J})$ of orientation preserving bijections is denoted by $\mathcal{C}^+(\mathfrak{J})$.

\begin{ex}
 Every countable group embeds into $\mathcal{C}(\mathfrak{J})$.
\end{ex}
\begin{ex}
Every countable left-orderable group embeds into the group of orientation preserving homeomorphisms of $\mathfrak{J}$, and therefore into $\mathcal{C}^+(\mathfrak{J})$.
\end{ex}

Since the set of non-dyadic rational points of $\mathfrak{J}$ is dense in $\mathfrak{J}$, the next lemma is a basic property of the (piecewise) continuity.
\begin{lem}
\label{lem-restriction-to-rationals}
Every function in $\mathcal{C}(\mathfrak{J})$ is uniquely determined by its values on non-dyadic rational points on $\mathfrak{J}$. Moreover, every function from $\mathcal{C}(\mathfrak{J})$ is continuous at non-dyadic rational points.\qed
\end{lem}

To construct respective embeddings into finitely generated simple groups we propose the following extension of the construction in \cite{mattebon_groups_2018}.

\subsection{Groups of flows of piecewise homeomorphisms}
Let us fix a subgroup $G$ of $ \mathcal{C}(\mathfrak{J})$.

\begin{df}\label{D: elements of T(G,Phi)}

 Define $ \tau_{\Sigma, \mathfrak{J}}: G \rightarrow \B(\Sigma)$ as follows: for each $g \in G$, let $\tau_{\Sigma, \mathfrak{J}}(g): = g_{\Sigma, \mathfrak{J}}$, where $g_{\Sigma, \mathfrak{J}}$ is defined by
\begin{align*}
g_{\Sigma, \mathfrak{J}}: [x,t] \mapsto [x,g(t)], \mbox{ for all } t\in \mathfrak{J},
\end{align*}
and $g_{\Sigma, \mathfrak{J}}$ is the identity map elsewhere.
\end{df}

We extend Definition \ref{D: T(Phi)} as follows.

\begin{df}[The group $T(G,\varphi)$]
\label{D: T(G,Phi)}
Let $G $ be a subgroup of $ \mathcal{C}(\mathfrak{J})$. We define $T(G,\varphi)$ as the subgroup of $\B(\Sigma)$ generated by $\tau_{\Sigma, \mathfrak{J}}(G)$ and $T(\varphi)$.
\end{df}
\begin{lem}
\label{Lem: fin gen of the main group}	
The group $G$ embeds into $T(G,\varphi)$ by $g \mapsto g_{\Sigma, \mathfrak{J}}$. Moreover, if $G$ is finitely generated, then $T(G,\varphi)$ is finitely generated as well.
\end{lem}
\begin{proof}
The second statement follows from the definition of $T(G,\varphi)$  and the fact that $T(\varphi)$ is finitely generated. For the first statement, it is enough to notice that, by definition, $g_{\Sigma, \mathfrak{J}}$ is an identity map if and only if $g=1$.
\end{proof}
\begin{df}[(Non-dyadic) rational points]
A point $[x, t] \in \Sigma$ is called a \emph{rational point} on $\Sigma$ if $t \in \mathbb{Q}$. If, in addition, $t$ is not dyadic, we say that $[x, t]$ is a\emph{ non-dyadic rational point}.
\end{df}
\begin{lem} 
\label{lem-dense-sets-of-rationals}
There exists a dense and recursive set of non-dyadic rational points in $\Sigma$.
\end{lem}
\begin{proof}
  Let us choose a recursive countable subset $\mathcal{X}:=\{x_1, x_2, \ldots \} \subset X$ that is dense in $X$, for example, the set of proper ternary fractions. Moreover, for all $i\in \mathbb{N}$, let $R_i \subset \Sigma$ be defined as
 $$ R_i :=\{[x_i, t]\mid \hbox{$t$ is rational and non-dyadic}\}.$$

We denote $R= \cup_{i=1}^{\infty} R_i.$
Note that each of $R_i$ is a recursive set.  Therefore, since $\mathcal{X}$ is also recursive by our choice, the we get that $R$ is recursive as well. 
\end{proof}

\begin{lem}
\label{L: unique extension for unifying groups}
If $G $ is  a subgroup of $ \mathcal{C}(\mathfrak{J})$, then the elements of $T(G,\varphi)$ are uniquely defined by their values on any (countable) dense set of  non-dyadic rational points of $\Sigma$. Moreover, the elements of $T(G,\varphi)$  are continuous at non-dyadic rational points of $\Sigma$.
\end{lem}
\begin{proof} By Lemma \ref{lem-dense-sets-of-rationals} there a fixed countable dense set of non-dyadic rational points in $\Sigma$. 
Let $R \subset \Sigma$ be such a set. Let us define $\mathcal{X} \subset X$ such that for each $x\in \mathcal{X}$ there exists $t\in \mathbb{Q}$ such that $[x, t] \in R$. Since $R$ is dense, $\mathcal{X}$ is dense as well.

Note that since $\mathcal{X} \subset X$ is dense in $X$, by Lemma \ref{L: orbits}, the set $\{[x, t] \mid x \in \mathcal{X}, t \in \mathbb{R}\}$ is dense in $\Sigma$. Therefore,  the elements of $T(G,\varphi)$ are uniquely defined by their restrictions to the $\Phi$-orbits of the elements $[x,0]$ for $x \in \mathcal{X}$. Now, the lemma follows from the combination of this observation with Lemma~\ref{lem-restriction-to-rationals}.
\end{proof}

\subsection{Simplicity and rigid stabilizers} \label{S: normal subgroups}

To prove simplicity results, we use the following standard tool.

Let $Y$ be a set, and $H$ a group acting faithfully on $Y$. Then the \emph{rigid stabilizer} of a subset $U\subset Y$ is the subgroup of $H$ whose elements move only points from $U$. We denote the rigid stabilizer of $U$ by $\R(U)$.

The following lemma is used to prove simplicity of $\T(\varphi)$ in \cite{mattebon_groups_2018}, cf. \cite[Lemma 2.1]{mattebon_groups_2018}.

\begin{lem}\label{L: epstein} 
Let $N$ be a normal subgroup of $H$. If there is a non-trivial element $g\in N$ and a non-empty subset $U\subset Y$  such that $g(U)\cap U = \emptyset$, then the first derived subgroup ${\R(U)}'$ is in $N$.
\qed
\end{lem}

  A group $G$ is called \emph{perfect} if it coincides with its first derived subgroup, that is $G=G'(=[G, G])$.

\begin{lem}
\label{lem: the unifying lemma}
If $G \leq \mathcal{C}(\mathfrak{J})$ is a perfect group, then $T(G,\varphi)$ is simple.
\end{lem}
\begin{proof}
Assume  that $N $ is a normal subgroup of $T(H,\varphi)$  and $N \neq\{1\}$. The proof of Lemma \ref{lem: the unifying lemma} follows from the following two claims.\\

\emph{Claim 1:} The group $\T(\varphi)$ is in $N$.

The proof of Claim 1 follows the arguments of simplicity in \cite{mattebon_groups_2018}.

\begin{proof}[Proof of Claim 1]

Let us fix a non-trivial element $g\in N$. Then, by Lemma \ref{L: unique extension for unifying groups}, there exists a non-dyadic rational point $y \in \Sigma$  such that $g(y)\neq y$.
By Lemma \ref{L: unique extension for unifying groups}, the elements of $T(G,\varphi)=\langle \tau_{\Sigma, \mathfrak{J}}(G), T(\varphi) \rangle $ are continuous at the non-dyadic rational points of $\Sigma$. Therefore, since $\Sigma$ is a Hausdorff space and $g(y) \neq y$, there exists an open neighborhood $U$ of $y$ such that $gU \cap U =\emptyset$. By Lemma \ref{L: epstein}, ${\R(U)}'$ is in $N$. 

Let $z\in \Sigma$, and choose $h\in \T(\varphi)$ such that $h(z)\in U$. Such a map $h$ exists as, by Lemma \ref{L: orbits}, the action of $\T(\varphi)$ on $\Sigma$ is minimal. Then, as $z\in h^{-1}U$ and $\R(h^{-1}U)=h^{-1} \R(U) h$, the rigid stabilizer $\R(h^{-1}U)$ is in $N$. 

As $h^{-1}U$ is open, there is a chart $\pi_{C, K}$ at $z$ which is in $h^{-1}U$. Since ${F_{C, K}}$ is in the rigid stabilizer of $h^{-1}U$, we conclude that ${F_{C, K}}' \subseteq N$. 

Therefore, for every chart $\pi_{C, K}$ there is a covering $\{C_i\times K_i\}$ of $C\times K$ such that ${F_{C_i, K_i}}'$ is in $N$. By Lemma \ref{L: coverings}, we conclude that for every chart $\pi_{C, K}$ the group ${F_{C, K}}'$ is in $N$. Now we use that $\bigcup_{I \subsetneq K} F_I \subset {F_K}'$ (cf. \cite[Theorem 4.1]{cannon_introductory_1996}) to conclude that the generators of $\T(\varphi)$ are in $N$.
\end{proof}

\emph{Claim 2:} For every normal subgroup $N$ of $\T(G,\varphi)$,   $\tau_{\Sigma, \mathfrak{J}}(G)$ is in $N$.
\begin{proof}[Proof of Claim 2]
Let $f\in F_{[0,1]}$ be an element of Thompson's group such that $J \cap f(J) = \emptyset$. Then $f_{X\times [0,1]}$ is in $\T(\varphi)$ and separates  $U_{X\times J}$ from $U_{X, f(J)}$. By the previous claim, $ f_{X, J}\in N$. Therefore, the first derived subgroup of the rigid stabilizer of the interior of  $U_{X\times J}$ is in $N$ by Lemma \ref{L: epstein}. Finally, we note that $\tau_{\Sigma, \mathfrak{J}}(G)$ is in the rigid stabilizer of the interior of $U_{X, J}$. Thus $\tau_{\Sigma, \mathfrak{J}}(G)'$ is in $N$. As $G$ is assumed to be perfect, this yields the claim.
\end{proof}
Now, to conclude the proof of Lemma \ref{lem: the unifying lemma}, we only need to combine the above claims with the fact that, by definition, $T(G,\varphi)=\langle \tau_{\Sigma, \mathfrak{J}}(G), T(\varphi) \rangle $.
\end{proof}

\subsection{Left-orders \texorpdfstring{on $T(G,\varphi)$}{}} \label{S: left-orders}

\begin{lem}
\label{lem: the unifying lemma-left order}
If $G \leq \mathcal{C}^+(\mathfrak{J})$, then the group $T(G,\varphi)$ is left-orderable. Moreover, if $G$ is finitely generated and consists of computable functions, then there exists a left-order on $T(G,\varphi)$ with recursively enumerable positive cone.
\end{lem}
\begin{proof}
First of all, note that since $G \leq \mathcal{C}^+(\mathfrak{J})$, the action of $T(G,\varphi)$  on $\Phi$-orbits of elements of $X \subset \Sigma$ is orientation preserving. 

Let $R=\{[a_1, s_1], [a_2, s_2], \ldots \}$ be a fixed, recursively enumerated and dense subset of non-dyadic rationals in $\Sigma$. The existence of such sets is by Lemma \ref{lem-dense-sets-of-rationals}.

Now for $f \in T(G,\varphi)$, define $f >1$ if for the smallest index $k \in \mathbb{N}$ such that $f([a_k, s_k]) \neq [a_k, s_k]$ we have $f([a_k, s_k])=[a_k, q_k]$ such that $q_k>s_k$. Therefore, by Lemma \ref{lem-restriction-to-rationals}, for all $f\neq 1$ either $f>1$ or $f^{-1}>1$ and for $f_1, f_2 >1$, $f_1f_2>1$. By Lemma \ref{L: unique extension for unifying groups}, the defined order is a left-order on $T(G,\varphi)$ .

Recall that the set $R$ is recursive. Therefore, to check whether $f>1$, we can consecutively compute the values $f([a_1, s_1]), f([a_2, s_2]), \ldots$. We stop at the first $k$ such that $f([a_k, s_k]) \neq [a_k, s_k]$. By Lemma \ref{L: unique extension for unifying groups}, this procedure stops if and only if $f \neq 1$, and since $G$ consists of computable maps, this procedure recursively enumerates the positive cone of the above defined left-order.
\end{proof}

\begin{cor}
\label{C: T(Phi) rec enum}
The group $T(\varphi)$  has a left-order with recursively enumerable positive cone.
\end{cor}


\section{Chart representations and the word problem in \texorpdfstring{$T(G,\varphi)$}{T(G,Phi)}}

Recall that $\mathfrak{J}$ is a fixed interval that is strictly contained in $[0,1)$. Let us fix a subgroup $G$ in $\mathcal{C}({\mathfrak{J}})$, and assume that $G$ is finitely generated and consists of computable functions. 

\subsection{Chart representations}

\begin{df}[Chart representations]\label{D: chart representation}  Let $h\in T(G,\varphi)$. A \emph{chart representation} of $h$ is a finite collection of triples 
$(C_i\times I_i, C_i\times J_i, h_i) $, where $h_i$ is a piecewise homeomorphism with countably many breakpoints on $I_i$ and $h_i(I_i)=J_i$, such that $\{U_{C_i\times I_i}\}$ and $\{U_{C_i\times J_i}\}$ cover $\Sigma$, and such that the restriction of $h$ to ${U_{C_i\times I_i}}$ is the function  $[x,t] \mapsto [x,h_i(t)]$. 

Each of the triples $(C_i\times I_i, C_i \times J_i, h_i) $ is called a \emph{chart}. The maps $h_i$ are \emph{local representations} of $h$. 
\end{df}

\begin{rem} Chart representations play the role of the \emph{(partial) tables} of Thompson in \cite{thompson_word_1980}.
\end{rem}

\begin{rem} By definition (compactness of $X$ and $\Sigma$) every $h\in T(G,\varphi)$ has a chart representation. 
\end{rem}

\begin{rem} Chart representations are not unique.
\end{rem}

\begin{ex}\label{Ex: chart} $f\in F_{X\times [-1/4,1/2]}$. We give two chart representations for $f$:
\begin{enumerate}
\item $\big\{(X\times  [-1/4,1/2], X \times [-1/4,1/2], f),$\\
 $(X\times [1/2,3/4], X\times [1/2,3/4], id)\big\}.$
\item $\big\{(X\times  [0,1/2], X \times f([0,1/2]), f|_{[0,1/2]})$\\  
$( X \times [3/4,1],X\times f( [-1/4,0])+1, t\mapsto f|_{[-1/4,0]}(t-1) +1)$,\\
$(X\times [1/2,3/4], X\times [1/2,3/4], id)\big\}$.
\end{enumerate}
\end{ex}

\begin{df}[$G$-dyadic maps]
\label{G-dyadic}
A piecewise homeomorphism $\Lambda:I_n \to I_0$ is a \emph{$G$-dyadic map} if $\Lambda=g_1f_1 \ldots g_nf_n$ is a composition of {piecewise homeomorphisms} $f_i:I_i\to J_i \subseteq \mathfrak{J}$ and $g_i:J_i \to I_{i-1}$, 
where all $f_i$ are dyadic maps, $f_i\not=id$ whenever $i\not=n$, and the $g_i$ are restrictions of non-trivial elements of $G$.  
\end{df}

\begin{rem} A $G$-dyadic map could a priori be equal to the identity map. 
\end{rem}

\begin{df}[Canonical chart representations]
\label{D: types and canonical representations}	
Let $h\in T(G,\varphi)$, and let  $\{(C_i\times I_i, C_i\times J_i,h_i)\}_{1\leqslant i\leqslant n}$ be a chart representation of $h$ such that for every $h_i$, $1\leq i \leq n$, one of the following takes place:
\begin{enumerate}
    \item[(I)] \label{A} $h_i$ is a dyadic map on $I_i$;
    \item[(II)] \label{B}  $h_i=f\Lambda$ is the composition a dyadic map $f$ and a $G$-dyadic map $\Lambda$. 
\end{enumerate}
Then, the representation $\{(C_i\times I_i, C_i\times J_i,h_i)\}_{1\leqslant i\leqslant n}$ is called \emph{canonical}. Also, charts for which $h_i$ corresponds to (I) or  (II) are called charts of type  (I)  or (II), respectively.
\end{df}

\subsection{Operations on charts}

The following operations can be applied to go from one chart representation to another.
\begin{df}[Inverse]\label{D: inverse}
Let $\{(C_i\times I_i, C_i\times J_i,h_i)\}_{1\leqslant i\leqslant n}$ be a chart representation of $h \in T(G, \varphi)$. Then  $\{(C_i\times J_i, C_i\times I_i,h^{-1}_i)\}_{1\leqslant i\leqslant n}$ is a chart representation of $h^{-1}$ and is called the \textit{inverse} of the initial one.
\end{df}

\begin{df}[Refinements]\label{D: refinement} The following operations on charts are called \emph{refinement}.
\begin{enumerate}
\item If $I=I_0\cup I_1$, then $(C\times I,C\times J, f)$ can be replaced by $(C\times I_0,C\times f(I_0), f|_{I_0})$ and 
$(C\times I_1,C\times f(I_1), f|_{I_1})$.
\item If $J=J_0\cup J_1$, then $(C\times I,C\times J, f)$ can be replaced by $(C\times f^{-1}(J_0),C\times J_0, f|_{f^{-1}(J_0)})$ and 
$(C\times f^{-1}(J_1),C\times J_1, f|_{f^{-1}(J_1)})$.
\item If $C=C_0\cup C_1$, then $(C\times I,C\times J, f)$ can be replaced by
$(C_0\times I,C_0 \times J, f)$ and $(C_1\times I,C_1 \times J, f)$.
\end{enumerate} 
\end{df}

\begin{df}[Reunions] \label{D: reunion} A \emph{reunion} is the inverse operation of a refinement.
\end{df}

\begin{df}[Shifts] \label{D: shift}  A \emph{shift} (of order $m\in \mathbb{Z}$) is replacing a triple $(C\times I,C\times J, f)$ by $(\varphi^{m} (C)\times (I-m),\varphi^m (C)\times (J-m), t\mapsto f(t+m)-m)$. 
\end{df}
\begin{rem}
A chart representation obtained by a refinements, reunion, or a shift on its charts corresponds to the same element from $T(G, \varphi)$. In particular, chart representations of elements from $T(G, \varphi)$ are not unique.
\end{rem}
\begin{rem} Since the  functions in $G$ are  computable, the operations \ref{D: inverse}, \ref{D: refinement}, \ref{D: reunion} and \ref{D: shift} are computable. 
\end{rem}

\begin{lem}\label{L: canonical generators}  $T(G,\varphi)$ is finitely generated and each of the generators can be represented by a canonical chart representation, which can be algorithmically determined.
\end{lem}
\begin{proof} Indeed, the generators of $T(\varphi)$ given by Lemma \ref{L: finite generation  of T(Phi)} can be represented as in Example \ref{Ex: chart}. One then applies a finite number of chart refinements at the breakpoints of the generating piecewise dyadic maps to obtain a canonical chart representation. The generators of $G$ can be represented by a canonical chart representation by definition. 
\end{proof}

\begin{lem}\label{L: canonicity and operations} The inverse, refinements and shifts preserve the canonicity of chart representations.  
\end{lem}
\begin{proof}
We will prove only that shift operations on charts of type (II) preserve the canonicity of chart representations, as the rest of statements of the lemma are straightforward.

  Suppose that the initial chart of type (II), on which a shift operation of order $m$ is applied, is $(C_i \times I_i, C_i \times J_i, \Lambda)$. Then a shift of order $m$ would transform it into the chart $(\varphi^{m}(C_i)\times (I_i-m), \varphi^{m}(C_i)\times (J_i-m),  \tilde{\Lambda} )$, where $\tilde \Lambda: I_i-m \rightarrow J_i-m$ is defined as $\tilde \Lambda(x)=\Lambda(x+m)-m$. 

Suppose that $\Lambda = fg_1f_1 \ldots g_nf_n$ is decomposed as in Definition \ref{G-dyadic}. Then, $\tilde\Lambda= \tilde f g_1f_1 \ldots g_n \tilde f_n$, where $\tilde f_n (x)= f_n(x+m)$ and $\tilde f (x)= f(x)-m$. The chart $\tilde \Lambda$ is also a $G$-dyadic map. Therefore, $(\varphi^{m}(C_i)\times (I_i-m), \varphi^{m}(C_i)\times (J_i-m),  \tilde{\Lambda} )$ satisfies the definition of charts of type (II) from Definition \ref{D: types and canonical representations}. Thus shift operations applied on charts of type (II) of canonical chart representations preserve the canonicity.
\end{proof}


\begin{df}[Composition]
\label{def-composition}
Let $\{(C_i\times I_i, C_i\times J_i,f_i)\}_{1\leqslant i\leqslant n}$  and  $\{(C'_i\times I'_i, C'_i\times J'_i,f'_i)\}_{1\leqslant i\leqslant m}$ be chart representations such that $ \bigcup I_i =\bigcup J_i' =[0, 1]$. Then we say that the chart representation 
$$  \{(C_{i,j}\times I_{i,j}, C_{i,j}\times J_{i,j}, f_{i,j})\}_{1\leqslant i\leqslant mn},$$ where
$$I_{i,j} = f_i'^{-1}(J_i'\cap I_j)\subseteq I'_i,~ J_{i,j}=f_j(f_i'^{-1}(J_i'\cap I_j))\subseteq J_j,$$   $$C_{i,j}= C_i \cap C_j, \text{~and~} f_{ij}=f_j\mid_{J_i'\cap I_j} \circ f_i'\mid_{f_i'^{-1}(J_i'\cap I_j)},$$
is their composition.
\end{df}
\begin{rem} \label{remark-on-compositions}
Note that if, in Definition \ref{def-composition}, the chart representations correspond to $f, f' \in T(G, \varphi)$, respectively, then the composition chart representation corresponds to $ff'$.
\end{rem}

\begin{rem}
\label{rem-composition-preserves-canonicity}
Note that if the two chart representations in Definition \ref{def-composition} are canonical, then their composition is canonical as well. In addition, finding the composition is a computable procedure.
\end{rem}

\begin{lem}\label{L: composition 1} Let $h\in T(G,\varphi)$ be given by a canonical chart representation 
$\{(C_i\times I_i, C_i,\times J_i,h_i)\}_{1\leqslant i\leqslant n}.$ 
Then there is an algorithm to determine a canonical chart representation 
$\{(C_j'\times I_j', C_j'\times J_j',h_i')\}_{1\leqslant j\leqslant n'}$  
of $h$ such that $\bigcup J_j' =[0,1]$. 
%
\end{lem}

\begin{proof} We describe the algorithm. For $1\leqslant i\leqslant n$, 

 if $J_i\subset [0,1]$ do nothing, go to $i+1$.
 
 if $J_i\cap [0,1]$ and $J_i\setminus (0,1)$ is non-empty, let $J_{i_1}=J_i\cap [0,1]$, $J_{i_2}=J_i\setminus (0,1)$ and apply a refinement (Definition \ref{D: refinement} (2)). Repeat from the beginning.

if $J_i\cap [0,1]$ is empty, determine $m$ such that $J_i-m \cap [0,1]$ is non-empty and apply a shift of order $m$ (Definition \ref{D: shift}). Repeat from the beginning. 

Since this procedure does not affect charts of type (2) from Definition \ref{D: types and canonical representations}, by Lemma \ref{L: canonicity and operations}, the canonicity of the initial chart representation is preserved.
\end{proof}
\begin{lem}\label{L: composition 2} Let $h\in T(G,\varphi)$ be given by a canonical chart representation 
$\{(C_i\times I_i, C_i,\times J_i,h_i)\}_{1\leqslant i\leqslant n}.$  
Then there is an algorithm to determine a  canonical chart representation 
$\{(C_j'\times I_j', C_j'\times J_j',h_i')\}_{1\leqslant j\leqslant n'}$  
of $h$ such that $\bigcup I_j' =[0,1]$. 
\end{lem}
\begin{proof} The proof is analogous to the proof of Lemma \ref{L: composition 1}.\end{proof} 

\begin{lem}
\label{lem-composition-is-computable}
There is an algorithm that for any two elements $f, f' \in T(G, \varphi)$, given by their canonical chart representations, computes a canonical representation of $ff'$.
\end{lem}
\begin{proof}
  By Lemmas \ref{L: composition 1} and \ref{L: composition 2}, there exists an algorithmic procedure that computes canonical chart representations $\{(C_i\times I_i, C_i\times J_i,f_i)\}_{1\leqslant i\leqslant n}$  and  $\{(C'_i\times I'_i, C'_i\times J'_i,f'_i)\}_{1\leqslant i\leqslant m}$ of respectively $f$ and $f'$ such that $ \bigcup I_i =\bigcup J_i' =[0, 1]$. Then their composition will be a canonical chart representation of $ff'$ (see Remarks \ref{remark-on-compositions} and \ref{rem-composition-preserves-canonicity}.)
\end{proof}

From the previous two lemmas we get:

\begin{lem}
\label{lem-computablility-of-chart-rep}
There exists an algorithm that for any input $f\in T(G, \varphi)$, given as a word in finite set of generators, outputs a canonical chart representation of $f$. In particular,  every element from  $T(G, \varphi)$ has a canonical chart representation.
\end{lem}
\begin{proof}
  It follows from Remark \ref{remark-on-compositions}, Lemma \ref{lem-composition-is-computable}, and the fact that the standard generators of $T(G, \varphi)$ have canonical chart representations, see Lemma \ref{L: canonical generators}.
\end{proof}

\subsection{The word problem}

The following observations are useful for studying the groups $T(G,\varphi)$.

\begin{lem}\label{L: id} Let $h\in T(G,\varphi)$ and let $\{(C_i\times I_i, C_i\times J_i,h_i)\}_{1\leqslant i\leqslant n}$ be a chart representation of $h$. 
Then $h=1$ in $T(G,\varphi)$ if and only if, for all $1\leqslant i\leqslant n$, we have $h_i=id$ (and $J_i=I_i$).
\end{lem}

\begin{proof}
Let $1\leqslant i\leqslant n$. Recall that $h$ maps $[x,t]$ to $[x,h_i(t)]$ for $(x,t) \in C_i\times I_i$. As $h=1$, 
$[x,h_i(t)]=[x,t].$ Thus, for any $x\in C_i$ and $t\in I_i$, there is $m\in \mathbb{Z}$ such that $(\varphi^m(x), t-m)=(x,h_i(t))$. We conclude that $h_i(t)=t-m$ and $\varphi^m(x)=x$. But $\varphi$ is a minimal subshift, that is, every orbit of $\varphi$ is dense. In particular, $m=0$. Therefore $h_i=id$. This yields one side of the assertion. The inverse assertion is trivial.
\end{proof}

\begin{lem}
\label{lemma-word problem in T-H-Phi} If 
 there is an algorithm to decide whether a $G$--dyadic map is equal to the identity,  
then the word problem in $T(G,\varphi)$ is decidable.
\end{lem}

\begin{proof}
By Lemma \ref{lem-computablility-of-chart-rep}, for every $h\in T(G,\varphi)$ one can algorithmically find a canonical  chart representation for $h$. By Lemma \ref{L: id}, $h=1$ if and only if for any canonical chart representation of $h$ the corresponding charts of types (I), and (II) are identity charts. If  $h_i$ a local representation in a chart of type (I), $h_i$ is a piecewise dyadic map with finitely many breakpoints, so that we can algorithmically check whether $h_i=id$. 

Now suppose that $h=f\Lambda:I\to I$ is a local representation in a chart of type (II), where $f:I_0\to I$ is dyadic and $\Lambda:I_n\to I_0$ a $G$-dyadic map. We note that $h=id$ on $I$ if, and only if, $\Lambda f = id $ on $I_0$. In fact, $\Lambda f$ is a $G$--dyadic map, so that, by assumption, we can algorithmically check whether $h=id$. 
\end{proof}

\begin{cor}\label{C: word problem T(Phi)} $T(\varphi)$ is computably left-orderable.  In particular,  the word problem in $T(\varphi)$ is decidable.
\end{cor}

\begin{proof} By Lemma \ref{lemma-word problem in T-H-Phi}, the word problem is decidable. By Corollary \ref{C: T(Phi) rec enum}, $T(\varphi)$ has a recursively enumerable positive cone. Thus, by Lemma \ref{lem: computable left-orders}, $T(\varphi)$ is computably left-orderable. 
\end{proof}

 \section{Embeddings into perfect groups} \label{S: embedding into perfect}

 Our next goal is to prove the following.

\begin{thm}\label{T: splinter} Every countable  group $G$ embeds into a finitely generated perfect group $H$. In addition, 
\begin{enumerate}
\item if $G$ is computable, then  $H$ has decidable word problem;
\item if  $G$ is left-ordered, then $H$ is left-ordered;
\item if $G$ is computably left-ordered, then the left order on $H$ is computable;
\item the embedding is a Frattini embedding.
\end{enumerate}
Moreover, in case (2) and (3), the order on $H$ continues the order on $G$.
\end{thm}

This strengthens \cite{neumann_embedding_1960}, \cite[Theorem 10A]{glass_ordered_1981} and \cite[Theorem 2.3]{thompson_word_1980}.
 If $G$ is assumed to be finitely generated, assertion $(1)$ is proved in \cite[Theorem 2.3]{thompson_word_1980}. Examples of (finitely generated) left-ordered and perfect groups are well-known, see \cite{bergman_right_1991}.

We first prove Theorem \ref{T: splinter} for finitely generated groups. In Section \ref{S: embedding into 2 generated}, we reduce the general case to the finitely generated case.

\subsection{Splinter Groups}  Let us assume that $G$ is a finitely generated group. We now construct a finitely generated perfect group in  which $G$ embeds. 
 Our construction resembles the splinter group construction of \cite[\S 2]{thompson_word_1980}. We comment on the construction of \cite{thompson_word_1980} in Section \ref{S: Thompson splinter group}.    
 
Let us fix an action of $T(\varphi)$ on the real line as follows: let us fix $z_0:=[x_0,0] \in \Sigma$. As the action of $\T(\varphi)$ on $\Sigma$ preserves the $\Phi$--orbits, $T(\varphi)$ acts on the $\Phi$--orbit of $z_0$, the action is orientation-preserving, and its orbits are dense. Finally, recall that the $\Phi$--orbit of $z_0$ is homoemorphic to $\mathbb{R}$. We fix such a homeomorphism. This induces an action of $T(\varphi)$ on $\mathbb{R}$. We fix this action of $\T(\varphi)$. 

Let $\mathcal{C}_0(\mathbb{R}, G)$ denote the group of functions from $\mathbb{R}$ to $G$ of bounded support. The action of $T(\varphi)$ on $\mathbb{R}$ induces an action $\sigma$ of $T(\varphi)$ on $\mathcal{C}_0(\mathbb{R} , G)$ such that for every $h\in \mathcal{C}_0(\mathbb{R} , G)$ and $f \in T(\varphi)$,  $$\sigma(f)(h)(s):=h(f^{-1}(s)).$$

The \emph{permutational wreath product} $G \wr_{\mathbb{R}} \T(\varphi)$ is defined as the semi-direct product $\mathcal{C}_0(\mathbb{R}, G) \rtimes_{\sigma} \T(\varphi)$, where, for $(h_1,f_1)$ and $(h_2,f_2) \in  \mathcal{C}_0(\mathbb{R}, G) \rtimes_{\sigma} \T(\varphi)$, $(h_1,f_1)(h_2,f_2):= (h_1 \sigma(f_1) (h_2),f_1f_2)$.

For every $g\in G$, we define the following function $\overline{g}$ in $\mathcal{C}_0(\mathbb{R}, G)$: 
\begin{align*}
\overline{g}(s):= \begin{cases} 
      g & \hbox{for } s\in [1/2,1), \\
      1 & \hbox{otherwise;}
   \end{cases}
\end{align*}
 and  $\overline{G}:=\{\overline{g}\mid g\in G\}$.

\begin{df}[Splinter groups]\label{D: Splinter group} The \emph{splinter group} is the subgroup of the permutational wreath product  $G \wr_{\mathbb{R}} \T(\varphi)$ generated by $\overline{G}$ and $\T(\varphi)$. We denote it by $\s(G,\varphi)$.
\end{df}

 Recall that $\T(\varphi)$ and $G$ are finitely generated. We note the following.

\begin{lem}\label{L: embedding splinter}
The group $G$ embeds into $\s(G,\varphi)$ with image $\overline{G}$. Moreover, $\s(G,\varphi)$ is finitely generated.
\qed
\end{lem}

\begin{lem}\label{L: perfect splinter} The splinter group $\s(G,\varphi)$ is perfect.
\end{lem}

To prove Lemma \ref{L: perfect splinter}, we adapt the arguments of the proof of \cite[Theorem 2.3]{thompson_word_1980}. We first prove

\begin{lem}\label{L: perfect splinter 1} The group $\overline{G}$ is in the first derived subgroup $\s(G,\varphi)'$ of $\s(G,\varphi)$.
\end{lem}

\begin{proof} Let $f_1$ and $f_2 \in T(\varphi)$ be such that $f_1$ maps $[1/2,1)$ onto $[1/4,1)$ and $f_2$ maps $[1/2,1)$ onto $[1/4,1/2)$, respectively. The existence of such elements follows from the definition of $T(\varphi)$. We note that for any $\bar{g} \in \bar{G}$
\begin{align*}
f_1\overline{g} f_1^{-1} (s) =\sigma(f_1)(\bar{g})(s)= & \begin{cases} 
      g & \hbox{for } s\in [1/4,1), \\
      1 & \hbox{otherwise;}\\
   \end{cases} 
   \\
   f_2 \overline{g} f_2^{-1} (s) =\sigma(f_2)(\bar{g})(s)= & \begin{cases} 
      g & \hbox{for } s\in [1/4,1/2), \\
      1 & \hbox{otherwise.}
   \end{cases} 
\end{align*}

Therefore, $\overline{g}=f_1\overline{g} f_1^{-1} (f_2\overline{g} f_2^{-1} )^{-1}$, hence $\overline{g}$ is a commutator element. This completes the proof.
\end{proof}

\begin{proof}[Proof of Lemma \ref{L: perfect splinter}]
Since $\T(\varphi)$ is simple, it is in $\s(G,\varphi)'$. By Lemma \ref{L: perfect splinter 1}, $\overline{G}$ is in  $\s(G,\varphi)'$ as well. Therefore,  $\s(G,\varphi)'=\s(G,\varphi)$. 
\end{proof}

\begin{lem}
\label{lem-isometry-splinter gp}
The group $G$ isometrically embeds into $\s(G,\varphi)$. 
\end{lem}
\begin{proof} Let $X$ and $Y$ be finite generating sets of $G$ and $T(\varphi)$, respectively. We prove that the embedding of $G=\langle X \rangle$ into 	$\s(G,\varphi) = \langle \bar{X} \cup {Y} \rangle$ by $g \mapsto \bar{g}$ is an isometric embedding, where $\bar{X}$ is the image of $X$ in $\s(G, \varphi)$.

Let $g \in G$. Also, let $f_i \in T(\varphi)$ and $g_i \in G$, $1\leq i \leq n$, be such that
 $$ \bar g =f_1\bar g_1 \ldots f_n \bar g_n$$
 and $| \bar g |_{\bar{X} \cup {Y}} = \sum_{i=1}^n |f_i|_{\bar{X} \cup {Y}} + \sum_{i=1}^n |\bar g_i|_{\bar{X} \cup {Y}}$, where $|\cdot|$ is the length of the group element with respect to the corresponding generating set.
 We have 
 $$\bar g = \bar{g_1}^{h_1}\bar{g_2}^{h_2} \ldots \bar{g_n}^{h_n} h_n,$$
 where $h_i=f_1 \ldots f_i$, $1\leq i \leq n$. Therefore, it must be that $h_n=1$ and 
 $$	\bar{g}(1/2) =g= \prod_{i \in I} g_i,$$
 where $I \subseteq \{1, \ldots, n \}$ is the set of indexes $i$ such that $h_i(1/2) \in [1/2, 1)$. Thus we get 
 $$\bar{g} = \prod_{i \in I} \bar g_i.$$
 Therefore, since we have $| \bar g |_{\bar{X} \cup {Y}} = \sum_{i=1}^n |f_i|_{\bar{X} \cup {Y}} + \sum_{i=1}^n |\bar g_i|_{\bar{X} \cup {Y}}$, we get $f_1=\ldots =f_n=1$ and $I= \{1, \ldots, n\}$, which implies that $|g|_X = |\bar{g}|_{\bar{X} \cup {Y}}$. Since $g$ is an arbitrary element of $G$, the last conclusion finishes the proof.
 
\end{proof}

\begin{lem}
\label{lem-frattini-splinter gp}
The embedding of $G$ into 	$\s(G,\varphi)$ by $g \mapsto \bar{g}$ is a Frattini embedding.
\end{lem}
\begin{proof}
	Let $g, h \in G$, and suppose that $\bar{g}$ and $\bar{h}$ are conjugate in $\s(G,\varphi)$. We want to show that $g$ is conjugate to $h$ in $G$.
	
There exist $e \in C_0(\mathbb{R}, G)$ and $f \in T(\varphi)$ such that $(e, f) \bar{g} (e, f)^{-1} = \bar{h}$ or, equivalently, $e \sigma(f)(\bar{g}) e^{-1}=\bar{h}$. Therefore, we have
	\begin{align*}
		\supp(e \sigma(f)(\bar{g}) e^{-1}) = \supp(\bar{h}) = [1/2, 1).
	\end{align*}
	 On the other hand, $$\supp(e \sigma(f)(\bar{g}) e^{-1}) = \supp(\sigma(f)(\bar{g}))=f(\supp(\bar{g}))=f([1/2, 1)).$$
	Thus $f([1/2, 1))=[1/2, 1)$, which implies that $\sigma(f)(\bar{g}) =\bar{g}$ (recall that $\bar{g}$ is constant on $[1/2, 1)$) and, hence, $e \bar{g} e^{-1} = \bar{h}. $ The last equality immediately implies that $g$ is conjugate to $h$ in $G$.
\end{proof}
 
\subsection{The word problem for \texorpdfstring{$\s(G,\varphi)$}{splinter groups}.} \label{S: word problem for splinter}

We recall that $\T(\varphi)$ is computably left-ordered, acts order-preservingly on $\mathbb{R}$, and that this action is computable.

We adapt a notion of splinter table introduced in \cite[p. 413]{thompson_word_1980}.
\begin{df}[Splinter table] A \emph{splinter table} corresponding to the element $(t,f)\in \s(G,\varphi)$ is a finite tuple of the form $(J_1, \ldots, J_n; g_1, \ldots, g_n; f)$, where $J_1, \ldots, J_n$ is a disjoint  finite collection of  bounded intervals from $\mathbb{R}$ whose union contains the support of $t: \mathbb{R} \rightarrow G$ such that $t(J_i)=g_i \in G$. 
\end{df}
\begin{ex} The group $\overline{G}$ coincides with the set of all splinter  tables $([1/2,1); g; 1_{T(\varphi)})$,  $g\in G$, and $\T(\varphi)$ coincides, for example, with the set of all splinter tables $([1/2,1); 1_G; f)$, $f\in \T(\varphi)$.   
\end{ex}

\begin{lem}
\label{L: splinter table} 

If $(t,f), (s,e) \in \s(G,\varphi)$ are given by their splinter tables, then their product $(t,f)(s,e)$ can be represented by a splinter table. The product splinter table can be computably determined. 
\end{lem}

\begin{proof} Suppose that the splinter tables of $(t,f)$ and  $(s,e)$ correspondingly are $$(J_1, \ldots, J_n; g_1, \ldots, g_n; f) \mbox{~and~} (I_1, \ldots, I_m; h_1, \ldots, h_m; e).$$

 Let $J:=\bigsqcup_{1\leqslant i \leqslant n} J_i$ and $I:=\bigsqcup_{1\leqslant j \leqslant m} I_j$.

Let $(r,q):=(t,f)(s,e)$. Then $q=fe$, and $r=t\, \sigma(f)s$ is a step function such that for all $1\leqslant i \leqslant n$ and for all $1\leqslant j \leqslant m$  
\[r\left( J_i \cap f(I_j)\right) = g_ih_j,\; r\left( J_i \setminus f(I)\right) = g_i,\; r\left(f(I_j)\setminus  J \right) = h_j\]
and the identity elsewhere. 

By the properties of $\T(\varphi)$, the inverse of $f$ as well as $ J_i \cap f(I_j)$, $J_i \setminus f(I)$ and $f(I_j)\setminus  J $ can be computably determined.
\end{proof}

\begin{cor}\label{C: splinter table} Every element of $\s(G,\varphi)$ can be represented by a splinter table. \qed
\end{cor} 

Note that $(J_1, \ldots, J_n; g_1, \ldots, g_n; f)$ is a splinter table corresponding to the trivial element of $\s(G,\varphi)$ if and only if $g_1=\ldots = g_n=1$ and $f=1$. Therefore, combining this observation and Lemma \ref{L: splinter table}  with the fact that the word problem of $T(\varphi)$ is decidable (Corollary \ref{C: word problem T(Phi)}),  we immediately get the following.

\begin{lem}\label{L: word problem splinter} If the word problem for $G$ is decidable, then so is the word problem for $\s(G,\varphi)$. \qed
\end{lem}

\subsection{Left-orders} \label{S: left-orders for splinter}

Now let $G$ be left-ordered. We then define a left-order on $\s(G,\varphi)$ as follows, cf. \cite{neumann_embedding_1960,glass_ordered_1981}: 
 let $(t,f) \in \s(G,\varphi)$ be given as a splinter table $(J_1, \ldots, J_n; g_1, \ldots, g_n; f)$, see Corollary \ref{C: splinter table}. 
 If $t\not = 1$, then, without loss of generality, we let $J_1$ be the leftmost interval such that $t(J_1)\not = 1$ (i.e. $g_1\neq 1$).  We set $(t,f)>1$ if
  either $f>1$ in $\T(\varphi)$ or if $f=1$ and $g_1>1$ in $G$. As the action of $T(\varphi)$ on $\mathbb{R}$ is orientation-preserving, this defines a left-order on  $\s(G,\varphi)$. 
    
 We conclude: 
\begin{lem}\label{L: left-order splinter} If $G$ is left-ordered, then so is $\s(G,\varphi)$. The order on $\s(G,\varphi)$ continues the order on $G$. \qed
\end{lem}

\begin{lem}\label{L: computably left-order splinter} If $G$ is computably left-ordered, then so is $\s(G,\varphi)$. The order on $\s(G,\varphi)$ continues the order on $G$. 
\end{lem}
\begin{proof} We fix a computable left-order on $T(\varphi)$, see Corollary \ref{C: word problem T(Phi)}. 
Let $(t,f)\in \s(G,\varphi)$. First run the algorithm for the word problem, see Lemma \ref{L: word problem splinter}. If $(t,f)$ represents the identity stop. Otherwise,  check whether or not $f$ is positive, negative or the identity. In the first two cases, we are done. Otherwise, we can computably determine the leftmost (maximal) interval $J$ of the splinter representation of $(t,f)$ such that $t(J)\not =1$. Then we use that the left-order on $G$ is computable to determine whether or not $t(J)$ is positive or negative.
\end{proof}

\subsection{Embeddings into finitely generated groups} \label{S: embedding into 2 generated}

To conclude the proof of Theorem \ref{T: splinter}, we need the following result of \cite{darbinyan_group_2015}, see also \cite[Theorem 3]{arman_new} for more details on assertions (1)-(3).
\begin{thm}\label{L: embedding into 2 generated} Every countable  group $G$  embeds into a $2$--generated group $H$. In addition, 
\begin{enumerate}
\item if $G$ is computable, then $H$ has decidable word problem;
\item if  $G$ is left-ordered, then $H$ is left-ordered;
\item if $G$ is computably left-ordered, then the left order on $H$ is computable; 
\item the embedding of $G$ into $H$ is a Frattini embedding.
\end{enumerate}
Moreover, the left-order on $H$ continues the left-order on $G$. \qed
\end{thm}

Here we briefly explain why the embedding from \cite{darbinyan_group_2015} is a Frattini embedding.  

\begin{proof}[Proof of assertion (4) of Theorem \ref{L: embedding into 2 generated}]
 As it is shown in Section 2 of \cite{darbinyan_group_2015}, for $G=\{g_1, g_2, \ldots\}$, the embedding satisfying Theorem \ref{L: embedding into 2 generated} has the following properties: it embeds $G$ into a two generated subgroup $\langle c, s \rangle$ of the group $G \wr \langle z \rangle \wr \langle s \rangle$, where $\langle z \rangle$ and $ \langle s \rangle$ are infinite cyclic groups, such that $g_i$ goes to $[c, c^{s^{2^i - 1}}] \in (G \wr \langle z \rangle)^{\langle s \rangle}$. Moreover, the element $[c, c^{s^{2^i - 1}}]$, regarded as a map ${\langle s \rangle} \rightarrow G \wr \langle z \rangle$, has support $\subseteq \{1\}$. In addition, $[c, c^{s^{2^i - 1}}](1)$ is a map $\langle z \rangle \rightarrow G$ such that $([c, c^{s^{2^i - 1}}](1))(1)= g_i$.
 
 Now assume that for two elements $g, h \in G$, their images in $\langle c, s \rangle $ are conjugate. Let  $\bar{g}$ and $\bar{h}$ be the images of $g$ and $h$ in $\langle c, s \rangle$, respectively. In particular, $\bar{g}$ and $\bar{h}$ are elements of the form $[c, c^{s^{2^i - 1}}]$. Let $(f, s^n) \in \langle c, s \rangle \leq (G \wr \langle z \rangle)\wr{\langle s \rangle}$ be such that $(f, s^n) \bar{g} (f, s^n)^{-1} = \bar{h}$. Then we get
 $$f \bar{g}^{s^n} f^{-1} = \bar{h},$$
 which implies that $\supp(f \bar{g}^{s^n} f^{-1}) =\supp(\bar{g})=\{1\}$. On the other hand,  
 $$\supp(f \bar{g}^{s^n} f^{-1}) = \supp (\bar{g}^{s^n})=\{s^{-n}\}.$$ Therefore, $n=0$, hence $\bar{g}(1)$ is conjugate to $\bar{h}(1)$ in $G \wr \langle z \rangle$. Repeating this argument one more time with respect to the pair $\bar{g}(1), \bar{h}(1) \in G \wr \langle z \rangle$ and using the fact that $(\bar{g}(1))(1)=g$ and $ (\bar{h}(1))(1)=h$, we get that $g$ is conjugate to $h$ in $G$. Since $g, h \in G$ are arbitrarily chosen elements from $G$, we get that the embedding from \cite{darbinyan_group_2015} that satisfies Theorem \ref{L: embedding into 2 generated} is Frattini.
\end{proof}

\begin{proof}[Proof of Theorem \ref{T: splinter}]
 By Theorem \ref{L: embedding into 2 generated}, we assume without loss of generality that $G$ is $2$--generated. 

Let $H$ be the splinter group $\s(G,\varphi)$. Then $G$  embeds into $H$ and $H$ is finitely generated by Lemma \ref{L: embedding splinter}. Moreover, $H$ is perfect by Lemma \ref{L: perfect splinter}. Assertion (1) follows from Lemma \ref{L: word problem splinter}. Assertion (2) follows from Lemma \ref{L: left-order splinter}. Assertion (3) follows from Lemma \ref{L: computably left-order splinter}. Assertion (4) follows from Lemma \ref{lem-frattini-splinter gp}.
\end{proof}

\subsection{Thompson's Splinter group revisited}\label{S: Thompson splinter group}
We compare Definition \ref{D: Splinter group} with Thompson's definition of a splinter group \cite[Definition 2.1]{thompson_word_1980}.

Let $X$ be a Cantor set, whose elements are represented as infinite  sequences in letters $0$ and $1$. We note that the so called Thompson's group $V$ is exactly the group $\ft(X)$ defined in \cite[p. 405]{thompson_word_1980}. In fact, $V$ is an infinite finitely generated simple group that acts on $X$ \cite[Proposition 1.5, Corollary 1.9]{thompson_word_1980}.  

We note that the splinter group of \cite{thompson_word_1980} is the subgroup of  $G\wr_X V$ generated by $V$ and the functions $\overline{g}$ from $X$ to $G$ that take the value $g$ on all sequences starting with $01$, and the identity elsewhere. Lemma \ref{L: perfect splinter} corresponds to \cite[Theorem 2.3]{thompson_word_1980}, and Lemma \ref{L: word problem splinter} to \cite[Proposition 2.7]{thompson_word_1980}.

Unfortunately, the group $V$ and, hence, the splinter group of \cite{thompson_word_1980} are not left-orderable.

\section{Embeddings of left-ordered groups}

Let $J$ be a dyadic interval in $[0,1]$. 
 Since every left-ordered group embeds as a subgroup into $\homeo^+(J)$, we have the following. 

\begin{prop}\label{P: main 1} Every countable left-ordered group $G$ embeds into a finitely generated left-ordered group $H$. In addition, the order on $H$ continues the order on $G$.   
\end{prop}

\begin{proof}
Let $G$ be countable left-orderable group. Then, by Theorem \ref{T: splinter}, $G$ embeds into a finitely generated perfect left-orderable group $G_1$. On its own turn, since $G_1$ is left-orderable, it embeds into $\homeo^+\left(J\right)$. Let $G_2 \leq \homeo^+\left(J \right)$ such that $G_2$ is isomorphic to $G_1$. Let $H=T(G_2,\varphi)$, see Definition \ref{D: T(G,Phi)}.  
By Lemmas \ref{Lem: fin gen of the main group}, \ref{lem: the unifying lemma} and \ref{lem: the unifying lemma-left order}, $H$  has the required properties. 
\end{proof}

We now construct an embedding as in the previous proposition, that, in addition, is Frattini  and isometric (provided that $G$ is finitely generated), as required by Remark \ref{IR: Frattini}, and that has the  computability properties required by Theorem \ref{main2}. To achieve this, we modify the construction of Proposition \ref{P: computable left order} of  embeddings of left-ordered groups into $\homeo^+(J)$.

\label{S: modified}
 
\subsection{Dyadic parts} 
\begin{df} 
	For any $r=2^k\frac{p}{q} \in \mathbb{Q}\setminus\{0\}$, where $p$ and $q$ are odd integers, we call  $\{r\}_d := k$ the \emph{dyadic part} of $r$. 
\end{df}

We observe:
\begin{lem}
\label{lem-mini}
Let $c\neq 0, \lambda, x \in \mathbb{Q}$ and $\{\lambda\}_d+\{x\}_d \neq \{c\}_d$. Then $\{\lambda x +c \}_d = \min \{ \{\lambda \}_d +\{x\}_d, \{c\}_d \}$.  Also, $\{ \lambda x \}_d= \{\lambda\}_d+\{x\}_d $. \qed
\end{lem}

\begin{df}\label{D: strong permutation}
Let $I$ and $J$ be fixed intervals and $g: \mathbb{Q}_I \rightarrow \mathbb{Q}_J$ be a bijection. Then we say that $g$ is \emph{strongly permuting the dyadic parts} if the following two conditions take place.
\begin{enumerate}
    \item For each $m \in \mathbb{Z}$, there exists at most one $x\in \mathbb{Q}_I$ such that $\{x\}_d=m$ and $\{g(x)\}_d \leq 0 $;
    \item If $x_1 \neq x_2 \in \mathbb{Q}_I$ and $\{x_1\}_d=\{x_2\}_d$, then $\{g(x_1)\}_d \neq \{g(x_2)\}_d $.
\end{enumerate}
\end{df}
If $g$ is a bijection from $I$ to $J$, when we say that $g$ is strongly permuting the dyadic parts if it maps rational points to rational points and its restriction $g\mid_{\mathbb{Q}_I}: \mathbb{Q}_I \rightarrow \mathbb{Q}_J$ satisfies Definition \ref{D: strong permutation}.
\begin{rem}
\label{remark-strongly-dyadic-permutations}
If  $g: \mathbb{Q}_I \rightarrow \mathbb{Q}_J$ is strongly permuting the dyadic parts, then, for each $m \in \mathbb{Z}$, the set $\left\{ \{g(x)\}_d \mid x\in \mathbb{Q}_{I}, ~ \{x\}_d=m\right\}$ is unbounded from above.
\end{rem}

Let us consider, for $0< i\leqslant n$, 
\begin{itemize}
\item bijective dyadic maps $f_i: I_{i}\to J_i$ such that  $ f_i(x)= \lambda_i x + c_i$, where $\lambda_i$ is a power of $2$ and, for all $i \notin \{0, n\}$, $c_i\not=0$; and
\item bijective maps $g_i:J_i\to I_{i-1}$, whose restriction to $\mathbb{Q}_{J_i}$ strongly permutes the dyadic parts.
\end{itemize}

\begin{lem}
\label{lemma-main technical}  If $\Lambda = g_1f_1g_2f_2 \ldots g_nf_n$, then, for large enough $N \in \mathbb{N}$, the set $$\left\{ \{\Lambda(x)\}_d \mid x\in \mathbb{Q}_{I_n}, ~ \{x\}_d=N\right\}$$ is unbounded from above. In particular, $\Lambda\not = id$. 
\end{lem}
\begin{proof}
We will prove the lemma by induction on $n$. 

Let $n=1$. Then $\Lambda(x) = g_1(\lambda_1 x + c_1)$. If $c_1=0$ or $N >\{c_1 \}_d$ and $\{x_1\}_d=\{x_2\}_d=N$, then, by Lemma \ref{lem-mini}, $\{\lambda_1 x_1 + c_1\}_d=\{\lambda_1 x_2 + c_1\}_d$. The statement now follows as $g_1$ strongly permutes the dyadic parts (see Remark \ref{remark-strongly-dyadic-permutations}).

Next let  $n>1$. Then $\Lambda(x) = g_1 (\lambda_1 \Lambda_2(x)+c_1)$, where $\Lambda_2 = g_2f_2 \ldots g_n f_n$ and $c_1 \neq 0$. By inductive assumption, for any large enough $N$, there exists a sequence $\{x_i\}_{i=1}^{\infty}$ such that $\{x_i\}_d=N$ and $\lim_{i\rightarrow \infty}\{\Lambda_2(x_i)\}_d=\infty$. By Lemma \ref{lem-mini}, for any large enough index $i$, $\{\lambda_1 \Lambda_2(x_i) + c_1\}_d=\{c_1\}_d$, hence, the lemma follows as $g_1$ strongly permutes the dyadic parts.
\end{proof}

\subsection{ The modified dynamical realization}
Let $J$ be a fixed closed interval in $\mathbb{R}$ with non-empty interior. We prove: 
\begin{prop}
\label{prop-modified-dyn-rel}
Let $G$ be a countable group. 

If $G$ is left-orderable, then there is an embedding $\Psi: G \hookrightarrow \homeo^+(J)$ such that, for all $g\in G \setminus \{1\}$, the map $\Psi(g): J \rightarrow J$ is strongly permuting the dyadic parts and does not fix any rational interior point of $J$.

If $G$ is computably left-orderable, then, in addition, all the maps $\Psi(g)$ can be taken to be  computable.
\end{prop}

 As in the proof of Proposition \ref{P: computable left order}, we fix a recursive enumeration $\mathbb{Q}_J=\{q_0, q_1, \ldots \}$ such that the natural order on $\mathbb{Q}_J$ is computable with respect to this enumeration. 

We first strengthen Lemma \ref{Lem: Cantor's argument} that states that there is an order preserving bijection $\Phi: G \to \mathbb{Q}_J$.

\begin{lem}
\label{Lem: Cantor's argument modified} If $G$ is enumerated and densly left-ordered, then
there is an enumeration $G=\{g_{i_1}, g_{i_2}, \ldots \}$ and an order preserving bijection $\Theta: G \to \mathbb{Q}_J$ such that 
\begin{enumerate}
\item For odd $j$, $\{\Theta(g_{i_j})\}_d \in \mathbb{N}$ and  $\{\Theta(g_{i_j})\}_d \not\in \left\{ \, \{\Theta(g_{i_k})\}_d \mid 1\leqslant k<j\right\}$;
\item For even $j$, $g_{i_j}\not \in \left\{ g_{i_k}g_{i_l}^{-1}g_{i_m}\mid 1\leqslant k,l,m<j \right\}$;
\item If $G$ is computably left-ordered, then the enumeration $G=\{g_{i_1}, g_{i_2}, \ldots \}$  and the map $j \mapsto \Theta(g_{i_j})$ are computable.
\end{enumerate}

\end{lem}
\begin{proof}
Let $G= \{1=g_0, g_1, g_2, \ldots \}$ and $\mathbb{Q}_J = \{0=r_0, r_1, r_2, \ldots \}$ be fixed (recursive)  enumerations. We define $\Theta: g_0 \mapsto r_0$ and $\Theta: g_{i_j} \mapsto r_{i_j}$, where $(g_{i_1}, g_{i_2}, \ldots )$ and $(r_{i_1}, r_{i_2}, \ldots )$ are permutations of  $(g_{1}, g_{2}, \ldots )$ and $(r_{1}, r_{2}, \ldots )$, respectively, defined recursively as follows. 

\emph{Step $2n+1$.} Let $G_{2n}=\{ g_{i_1}, \ldots, g_{i_{2n}}\}$ and $Q_{2n}=\{r_{i_1}, \ldots, r_{i_{2n}}\}$ be already defined. Let us define $g_{i_{2n+1}}$ as the element of the smallest index that is not in $G_{2n}$. Suppose that $g_{i_s}<g_{i_{2n+1}}< g_{i_t}$ and that no  element from $G_{2n}$ is in between $g_{i_s}$ and $ g_{i_t}$. Then define $r_{i_{2n+1}}\in \mathbb{Q}_J$ to be of the smallest index such that 
\begin{itemize}
    \item[(O1)] $r_{i_{2n+1}} \notin Q_{2n}$ and $r_{i_s} < r_{i_{2n+1}}<r_{i_t}$,
    \item[(O2)] $\{r_{i_{2n+1}} \}_d \in \mathbb{N}$ and $\{r_{i_{2n+1}} \}_d \notin \{  \{r_{i_{j}} \}_d \mid 1\leq j \leq 2n \}$.
\end{itemize}

\emph{Step $2n+2$.} Let $G_{2n+1}:=\{ g_{i_1}, \ldots, g_{i_{2n+1}}\}$ and $Q_{2n+1}=\{r_{i_1}, \ldots, r_{i_{2n+1}}\}$ be already defined. Let us define $r_{i_{2n+2}}$ as the rational of the smallest index that is not in $Q_{2n+1}$. Suppose that $r_{i_s}<r_{i_{2n+2}}< r_{i_t}$ and that no element from $Q_{2n+1}$ is in between $r_{i_s}$ and $ r_{i_t}$. Then let us define $g_{i_{2n+2}}\in G$ as the element of the smallest index such that 
\begin{itemize}
     \item[(E1)] $g_{i_{2n+2}} \notin G_{2n+1}$ and $g_{i_s} < g_{i_{2n+2}}<g_{i_t}$, and
    \item[(E2)] $g_{i_{2n+2}} \notin \{ g_{i_k} g_{i_l}^{-1} g_{i_m} \mid 1\leq k, l, m \leq 2n+1 \}$.
\end{itemize} 

The bijection $\Theta$ defined this way is order preserving by (O1) and (E1). Condition (O2) yields assertion (1), and (E2) yields assertion (2). Finally, as the procedure is algorithmic, we also obtain assertion (3).
\end{proof}

\begin{proof}[Proof of Proposition \ref{prop-modified-dyn-rel}] By Remark \ref{R: embedding in dense}, we may assume that the order on $G$ is dense.
Let the enumeration  $G=\{g_0, g_1, \ldots \}$ and $\Theta: G \to \mathbb{Q}_J$ satisfy the assertions of Lemma \ref{Lem: Cantor's argument modified}. Then $\Theta: G \to \mathbb{Q}_J$ induces the embedding $\rho_G^{\Theta}: G \hookrightarrow \homeo^+(J)$ according to $\rho_G^{\Theta}(g)(\Theta(h))= \Theta(gh)$ for $g, h \in G$.  Denote $\Psi=\rho_G^{\Theta}$. 
 Let  $h \in G\setminus \{1\}$. By Lemmas \ref{L: embedding computable left order} and \ref{L: dynamical realisation no fixed points}, we only need to show that $\Psi(h)$ is strongly permuting the dyadic parts.  
 
 To this end, we enumerate $\mathbb{Q}_J$ such that $\Theta(g_i)=r_i$ and let $r_i \neq r_j \in \mathbb{Q}_J$. We define $r_k=\Psi(h)(r_i)=\Theta(hg_i)$ and $r_l=\Psi(h)(r_j)=\Theta(hg_j)$, so that $g_k=hg_i$ and $g_l=hg_j$. 

We first show property (1) of Definition \ref{D: strong permutation}. 
By contradiction, assume that there exist $i\neq j$ such that $\{r_i\}_d = \{r_j\}_d$ and $\{r_k\}_d, \{r_l\}_d \notin \mathbb{N} $. 
Since $\{r_k\}_d, \{r_l\}_d \notin \mathbb{N}$, the indices $k$ and $l$ are even. 
 Then, since $g_k=hg_i = (hg_j)(g_j^{-1})(g_i)$ and $g_l=h g_j = (hg_i) (g_i^{-1})(g_j)$, by (2) of Lemma \ref{Lem: Cantor's argument modified}, the largest index is among $i$ or $j$. Let $j>i$. Then, since $\{r_i\}_d = \{r_j\}_d$, we get the index $j$ is even. Since  $g_j = g_i(hg_i)^{-1} (hg_j)$, again by (2) of Lemma \ref{Lem: Cantor's argument modified}, we get a contradiction, which yields the claim. 

Next, we prove property (2) of Definition \ref{D: strong permutation}. By contradiction, assume that there exist $r_i \neq r_j \in \mathbb{Q}_J$ such that $\{r_i\}_d=\{r_j\}_d$ and suppose that $\{r_l\}_d=\{r_k\}_d$. 
Without loss of generality,  $l>i,j,k$ (if, say, $j>i,k,l$, then instead of $h$ we could consider $h^{-1}$). 
Then, since $\{r_k\}_d=\{r_l\}_d$, by (1) of Lemma \ref{Lem: Cantor's argument modified}, $l$ has to be even. 
Therefore, since $\Theta(hg_j)=r_l$ and $l$ is even, by (2) of Lemma \ref{Lem: Cantor's argument modified}, $hg_j\not \in \left\{ g_mg_n^{-1}g_p\mid 1\leqslant m,n,p<j \right\}$. On the other hand, since $l>i,j,k$, we get $hg_j=(hg_i)(g_i^{-1})(g_j)$, a contradiction. 

This completes the proof of Proposition \ref{prop-modified-dyn-rel}.
\end{proof}

\subsection{The embedding theorems}

Let $G$ be countable left-orderable group. Then, by Theorem \ref{T: splinter}, $G$ embeds into a finitely generated perfect left-orderable group $G_1$. Moreover, this embedding is a Frattini embedding. 

Let $J:=[1/4,1/2]$. Since $G_1$ is left-orderable, there is an embedding $\Psi: G_1\hookrightarrow \homeo^+\left(J\right)$. Let $G_2 =\Psi(G_1)$.
By Proposition \ref{prop-modified-dyn-rel}, we can assume that the non-trivial elements of $G_2\subseteq \homeo^+\left(J \right)$ strongly permute the dyadic parts and do not fix any rational interior point of $J$. 

For the definition of $G_2$--dyadic maps, see Definition \ref{G-dyadic}.

\begin{lem}\label{L: no G dyadic map} 
Let $\Lambda$ be a $G_2$--dyadic map. If $G_2$ has decidable word problem, then there is an algorithm to decide whether or not $\Lambda=id.$
\end{lem}

\begin{proof} Let  $n>0$ and, for all $0\leqslant i \leqslant n $, let ${J}_i\subset J$ and let $g_i:{J_i} \rightarrow {I_{i-1}}$ be the restriction of an element of $G_2$ such that $g_i\not=id$. Moreover, let $f_i: {{I}_i} \rightarrow {{J}_i} $, given by $ f_i(x)= \lambda_i x + c_i$, be dyadic maps such that $f_i\not=id$.  

Since,  for all $0<i<n$, $J_i\subset [1/4,1/2]$ and, by definition, $\lambda_i$ is  a power of $2$, we get that $c_i\not=0$ for $0<i<n$.  
Then, by Lemma \ref{lemma-word problem in T-H-Phi}, $\Lambda:= g_1f_1g_2f_2 \ldots g_nf_n\not =id$.

If $n=1$ and $f_1=id$, then $\Lambda =g_1 \in G$. Then we decide using the algorithm for  the word problem in $G$. 
\end{proof}

Combining Lemmas \ref{lem-isometry-splinter gp} and \ref{L: no G dyadic map}, we also conclude the following. 
\begin{lem}\label{L: isometry-1} The embedding $ G_1  \hookrightarrow T(G_2,\varphi)$ is isometric.\qed \end{lem}

\begin{lem}\label{L: Frattini left order} The embedding $G_1 \hookrightarrow T(G_2,\varphi)$ is a Frattini embedding.
\end{lem} 

\begin{proof} Let $h,g \in G_2$ and $t\in T(G_2,\varphi)$. We assume that $ht^{-1}gt=1$. 

We represent $t$ by a canonical chart representation $(C_i\times I_i, C_j\times J_i, t_i)$ such that $\bigcup_{i}J_i=[0,1]$; and represent $t^{-1}$ by $(C_i\times J_i, C_i\times I_i, t_i)$. We recall that $t_i(I_i)=J_i$. 

Let $k$ be a index such that  $1/2$ is in the closure of $J_k$ and such that (after applying a chart refinement if necessary)  $J_k\subseteq J$. As $g$ is fixing $1/2$, there is $J_k'\subseteq J_k$ such that $g(J_k')\subseteq  J_k$ and such that  $1/2$ is in the closure of $J_k'$. We let $I_k'=t_k^{-1}(J_k')$ and $I_k''= t_k^{-1}gt_k(I_k')$.

Then the triple $(C\times I_k', C \times I_k'', t_k^{-1}gt_k)$ is in a chart representation of $t^{-1}gt$.  Up to applying the algorithm of Lemma \ref{L: composition 1} to this chart representation, we may assume that $I_k''$ is in $[0,1]$. Moreover, up to applying a chart refinement if necessary, we may assume that either $I_k''\cap \mathfrak{J}$ is empty or consists of one point ($1/4$ or $1/2$), or $I_k''\subseteq \mathfrak{J}$.

If $I_k''\cap J$ is empty or consists of one point, then $(C\times I_k', C \times I_k'', t_k^{-1}gt_k)$ is in a chart representation of $ht^{-1}gt$. Thus $t_k^{-1}gt_k=id$ on $I_k'$ by Lemma \ref{L: id}. This implies that $g$ acts as the identity on $J_k'$.  Since non-trivial elements of $G_2$ do not fix any rational interior points of $J$, the element $g=1$.

Otherwise, the triple $(C\times I_k', C \times h(I_k''), ht_k^{-1}gt_k)$ is in a chart representation of $ht^{-1}gt$. Thus $ht_k^{-1}gt_k=id$ on $I_k'$ by Lemma \ref{L: id}. Then $h(I_k'')=I_k'$. This is only possible if $I_k'\subseteq J$. But then Lemma \ref{L: no G dyadic map}, implies that $t_k$ acts as an element of $G_2$.  Since non-trivial elements of $G_2$ do not fix any rational interior points of $J$, this implies that $g$ and $h$ are conjugate in $G_1$. 
\end{proof}

We can now conclude Theorems \ref{main1}, \ref{main2} and \ref{theorem-bludov-glass}. 

\begin{proof}[Proof of Theorems 1,  2 and Remark \ref{IR: Frattini}] We let $H=T(G_2,\varphi)$.

The group $H$ is finitely generated left-orderable and simple by Lemmas \ref{Lem: fin gen of the main group}, \ref{lem: the unifying lemma-left order} and \ref{lem: the unifying lemma}. By construction, $G$ embeds into $H$, and the order on $H$ extends the order on $G$.  Moreover, by Lemma \ref{L: Frattini left order}, the embedding of $G$ is a Frattini embedding. 

If $G$ is computably left-ordered, we may in addition assume that $G_1$ is computably left-ordered, see Theorem \ref{T: splinter}. By Proposition \ref{prop-modified-dyn-rel}, for all $g\in G_1$, $\Psi(g)$ is computable. Therefore, by Lemma \ref{lem: computable left-orders}, the positive cone of $H$ is recursively enumerable. Moreover, by Lemma \ref{L: no G dyadic map} and Lemma \ref{lemma-word problem in T-H-Phi}, the group $H$ has decidable word problem. By Lemma \ref{lem: computable left-orders}, the left-order on $H$ is computable. 
\end{proof}

\begin{proof}[Proof of Theorem \ref{theorem-bludov-glass}] Let $G$ be finitely generated left-orderable group with a recursively enumerated positive cone. If $G$ has decidable word problem, then the left-order on $G$ is computable by Lemma \ref{lem: computable left-orders}. Then Theorem \ref{main2} implies that $G$ embeds into a finitely generated computably left-ordered simple group $H$. In particular, the word problem in $H$ is decidable. Thus $H$ can be defined by a recursively enumerable set of relations. By \cite[Theorem D]{bludov_word_2009}, $H$  embeds into a left-orderable finitely presented group. 

On the other hand, if $H$ is a finitely generated simple subgroup of a finitely presented group, then it has decidable word problem (see \cite[Theorem 3.6]{lyndon_schupp}). Therefore, $G$ has decidable word problem as well. 
\end{proof}

\section{Embeddings of computable groups}
\label{S: Thompson's theorem}

In this section we prove Theorem \ref{T: main 4}, the isometric version of Thompson's theorem \cite{thompson_word_1980}. In Appendix we present yet another proof of Theorem \ref{T: main 4} that, using the setting of our paper, mimics the original idea of \cite{thompson_word_1980}.

 \begin{thm}\label{T: thompson-1} Every computable group $G$ Frattini embeds into a finitely generated simple group $H$ with decidable word problem. If $G$ is finitely generated, then the embedding is isometric.
\end{thm}

\begin{rem}
The original statement \cite{thompson_word_1980} is for finitely generated groups, but finite generation can be replaced by computability of $G$ due to Theorem \ref{L: embedding into 2 generated}. 
\end{rem}

\subsection{The embedding construction}

Let $G$ be a computable group. By Theorem \ref{T: splinter}, $G$ embeds into a finitely generated perfect group $G_1$ with decidable word problem (if $G$ is finitely generated, this claim also follows from \cite[\S 2]{thompson_word_1980}). 

Let $G_1=\{ g^{(1)}, g^{(2)}, \ldots \}$  be enumerated so that $\mathfrak{m}: \mathbb{N} \times \mathbb{N} \rightarrow \mathbb{N}$,  defined as $\mathfrak{m}((i, j)) = k$ if $g^{(i)} g^{(j)} =g^{(k)}$, is computable. By Remark \ref{R: computable group and word problem} the existence of such $\mathfrak{m}$ is equivalent to decidability of the word problem. 

Let us fix two recursively enumerated recursive sets of dyadic numbers $\{ x_1, x_2, \ldots \}$ and $\{ y_1, y_2, \ldots \}$ such that the following takes place

\begin{enumerate}
	\item $0<x_1< y_1<x_2< y_2<  \ldots < \frac{1}{3}$,
	\item $x_i$ and $y_i$ are of the form $\frac{m}{2^n}$ and  $\frac{m+1}{2^n}$, respectively,
	\item $\lim_{i\rightarrow \infty} x_i = \frac{1}{3}$.
\end{enumerate}

Let us denote $D_i=[x_i, y_i]$ and $\mathfrak{J} = (0, \frac{1}{3}] \supset \sqcup_{i=1}^{\infty} D_i$.

For every $l\in \mathbb{N}$, let $\xi_l: \mathfrak{J} \rightarrow  \mathfrak{J}$ be such that, for every $k\in \mathbb{N}$, it is an affine map from  $D_k$ onto $D_{\mathfrak{m}(l,k)}$ and that is identity outside of $\cup_{i=1}^{\infty} D_i$. In particular, the map $\xi_l: D_k \rightarrow D_{\mathfrak{m}(l,k)}$ is dyadic.
\begin{rem}
	\label{rem-dyadic-action}
	 The maps $\xi_l: \mathfrak{J} \rightarrow  \mathfrak{J}$, $l\in \mathbb{N}$, are computable and continuous at non-dyadic points. Also, they have finitely many (dyadic) breakpoints outside of any open neighborhood of $1/3$.
\end{rem}

Let us define $\lambda: G_1 \rightarrow \mathcal{C}(\mathfrak{J})$  by $\lambda(g^{(l)})=\xi_l$, for all $l\in \mathbb{N}$. 

\begin{rem}\label{R: thompson} The map $\lambda$ is an embedding of $G_1$ into computable maps in  $\mathcal{C}(\mathfrak{J})$. 
\end{rem}
Let $\Lambda= g_1f_1g_2f_2 \ldots g_nf_n: \mathcal{I}_n \rightarrow \mathcal{I}_0$, where $f_i: {\mathcal{I}_i} \rightarrow {\mathcal{J}_i} $ and $g_i: \mathcal{J}_i \rightarrow \mathcal{I}_{i-1}$, be a $G_2$-dyadic map as in Definition \ref{G-dyadic}. Recall that in particular we have $\mathcal{J}_i \subseteq \mathfrak{J}=(0, \frac{1}{3}]$ for $1\leq i \leq n$.  We say that $\Lambda$ is a \emph{special} $G_2$-dyadic map if for each $1\leq i \leq n$ we have 
 $1/3 \in \overline{\mathcal{J}}_i$ (closure of $\mathcal{J}_i$).  Correspondingly, we say that a chart of type (II), see Definition \ref{D: types and canonical representations}, is $\emph{special}$ if the local representation in this chart is of the form $f_0\Lambda$, where $\Lambda$ is a special $G_2$--dyadic map.

The following lemma is a direct consequence of Remark \ref{rem-dyadic-action} and of the fact that the maps $f_i$, $g_i$, $1\leq i \leq n$, are computable.
\begin{lem}
	\label{lemma-easy-observation}
	There exist a finite  collection of intervals $K_1, K_2, \ldots, K_s \subseteq (0, \frac{1}{3}]$ such that $\Lambda|_{K_i \cap \mathcal{I}_0}$ is a special $G_2$-dyadic map. Moreover, such intervals $K_1, K_2, \ldots, K_s $ can be found algorithmically. \qed
\end{lem}

\begin{lem}
\label{lem-property-of-dyadic-maps}
	If $h: I \rightarrow J$ is a surjective dyadic map such that $1/3$ is in the closures of $I$ and ${J}$, and $ {I}, {J} \subseteq (0, \frac{1}{3}]$, then $h$ is the identity map.
	\end{lem}
\begin{proof}
	The lemma follows from the fact that $\frac{1}{3}$ is non-dyadic.
\end{proof}
By Lemma \ref{lem-property-of-dyadic-maps}, we have: 
\begin{cor}\label{C: special G dyadic} Special local representations are of the form $f_0g_1f_1$. In particular, if a special $G_2$-dyadic map or a special chart of type (II) fixes $1/3$, then it acts as an element of $G_2$. \qed
\end{cor}
%

Since the word problem for $G_2$ is decidable, this implies that  there exists an algorithm that decides whether or not a special $G_2$-dyadic map represents the identity map. This, combined with Lemmas \ref{lemma-easy-observation} and \ref{lemma-word problem in T-H-Phi}, leads to the following corollary.

\begin{cor}
\label{cor-wp-thompson-a}
		The word problem in $T(G_2, \varphi)$ is decidable. \qed
\end{cor}

	\subsection{Frattini property}
	The embedding $G \hookrightarrow T(G_2, \varphi)$ is Frattini.
	
	\begin{lem}\label{lem: Frattini arman embedding} The embedding $G_1 \hookrightarrow T(G_2,\varphi)$ is a Frattini embedding.
\end{lem} 
We adapt the proof of Lemma \ref{L: Frattini left order}.
\begin{proof}
Let $h,g \in G_2$ and $t\in T(G_2,\varphi)$. We assume that $ht^{-1}gt=1$. 

We represent $t$ by a canonical chart representation $\{(C_i\times I_i, C_j\times J_i, t_i)\}$ such that $\bigcup_{i}J_i=[0,1]$; and represent $t^{-1}$ by $\{(C_i\times J_i, C_i\times I_i, t_i^{-1})\}$. We recall that $t_i(I_i)=J_i$. 

Let $k$ be a index such that  $1/3$ is in the closure of $J_k$ and such that (after applying a chart refinement if necessary)  $J_k\subseteq \mathfrak{J}$. As $g$ is fixing $1/3$, there is $J_k'\subseteq J_k$ such that $g(J_k')\subseteq  J_k$ and such that  $1/3$ is in the closure of $J_k'$. We let $I_k'=t_k^{-1}(J_k')$ and $I_k''= t_k^{-1}gt_k(I_k')$. 

Then the triple $(C\times I_k', C \times I_k'', t_k^{-1}gt_k)$ is in a chart representation of $t^{-1}gt$.  Up to applying the algorithm of Lemma \ref{L: composition 1} to this chart representation, we may assume that $I_k''$ is in $[0,1]$. Moreover, up to applying a chart refinement if necessary, we may assume that either $I_k''\cap \mathfrak{J}$ is empty or consists of one point $1/3$, or $I_k''\subseteq \mathfrak{J}$.

If $I_k''\cap \mathfrak{J}$ is empty or consists of one point $1/3$, then $(C\times I_k', C \times I_k'', t_k^{-1}gt_k)$ is in a chart representation of $ht^{-1}gt$. Thus $t_k^{-1}gt_k=id$ on $I_k'$ by Lemma \ref{L: id}. This implies that $g$ acts as the identity on $J_k'$.  As $1/3$ is in the closure of $J_k'$, this implies that $g=1$.

Otherwise, the triple $(C\times I_k', C \times h(I_k''), ht_k^{-1}gt_k)$ is in a chart representation of $ht^{-1}gt$.  Thus $ht_k^{-1}gt_k=id$ on $I_k'$ by Lemma \ref{L: id}. But then $t_kht_k^{-1}g:J_k'\to J_k'$ has to be the identity as well. As $g$ is fixing $1/3$, $t_kht_k^{-1}$ has to fix $1/3$. 

If $t_k$ was dyadic (i.e. of type (I)), it would have to fix $1/3$, so that $t_k=id$ by Lemma \ref{lem-property-of-dyadic-maps}. If $t_k$ is of type (II), we may assume that $t_k$ is special, see Lemma \ref{lemma-easy-observation}. Then, by Corollary \ref{C: special G dyadic}, $t_k$ acts as an element of $G_2$.  

Thus $g$ and $h$ are conjugated by elements of $G_2$. This implies that $g$ and $h$ are conjugated in $G_1$. 
	\end{proof}
	
	Combining Lemma \ref{lem: Frattini arman embedding} with Lemma \ref{lem-frattini-splinter gp}, we obtain:
\begin{cor}
\label{corollary-frattini-thompson-a}
    The embedding $G \hookrightarrow T(G_2, \varphi)$ is Frattini. \qed
\end{cor}

	\begin{lem}\label{lem: isometry arman embedding} The embedding $G_1 \hookrightarrow T(G_2,\varphi)$ is an isometric embedding.
\end{lem} 	
\begin{proof} We fix a finite generating set $X$ for $G_1$, and denote the generating set of $T(\varphi)$ given by Lemma \ref{L: finite generation  of T(Phi)} by $Y$. We denote the union of the bijective images of $X$ and $Y$ in $T(G_2,\varphi)$ by $Z$ and recall that $Z$ generates $T(G_2,\varphi)$. We assume that all generating sets are symmetric. We denote by $|.|_A$ the word metric with respect to the generating set $A$.   

Let $g\in G_2$ and let $t=z_1\cdots z_m$ be a reduced word in the alphabet $Z$ that represents $g^{-1}\in T(G_2,\Phi)$, so that $tg=1$. In addition, we assume that $m=|g|_Z$.  We represent every generator $z_i$ by a canonical chart representation, see Lemma \ref{L: canonical generators}. Lemma \ref{lem-computablility-of-chart-rep} then gives a canonical chart representation  $(C_i\times I_i, C_i\times J_i, t_i)$ for $t$. Recall that the maps $t_i$ are compositions $t_i=h_1\cdots h_{m_i}$, where each map $h_j$ is a local representation in the canonical chart representation of a generator in $Z$ and $m_i\leqslant m$. In addition, up to applying the algorithm of Lemma \ref{L: composition 1} to this chart representation of $t$, we may assume that $\bigcup I_i=[0,1]$.  

Let $I_k$ be  an interval such that $1/3$ is in the closure of $I_k$ and such that (after applying a chart refinement if necessary) $I_k\subseteq \mathfrak{J}$. 

Then $(C_i\times g^{-1}(I_k), C_i\times J_i, t_ig)$ is in a canonical chart representation of the identity. By Lemma \ref{L: id}, $t_ig$ is the identity mapping. In particular, $g^{-1}(I_k)=J_i$, so that $J_i\subseteq \mathfrak{J}$ and $1/3$ is in the closure of $J_i$. 

We note that $t_i$ is not dyadic (i.e. of type (I)) by Lemma \ref{lem-property-of-dyadic-maps}.  
If $t_i=fg_1f_1\cdots g_nf_n$ is a chart of type (II), then, by Lemma \ref{lemma-easy-observation}, we may assume that $t_i$ is special. Thus  $t_i=g_1\cdots g_n \in G_2$ (Lemma \ref{lem-property-of-dyadic-maps}), where $n\leqslant m_i \leqslant m$. 

Thus we may assume that $t_i=x_{j_1}\cdots x_{j_{m_i}} \in G_2$. Then $|g|_X\leqslant m_i \leqslant m=|g|_Z$.  
We conclude that the embedding is isometric. 
\end{proof}
Combining Lemma \ref{lem: isometry arman embedding} with Lemma \ref{lem-isometry-splinter gp}, we obtain:
\begin{cor}
\label{corollary-isometry-thompson-a}
    If $G$ is finitely generated, then the embedding $G \hookrightarrow T(G_2, \varphi)$ is isometric. \qed
\end{cor}
\begin{proof}[Proof of Theorem \ref{T: thompson-1}] 
 The simplicity of $T(G_2, \varphi)$ follows from Lemma \ref{lem: the unifying lemma}. From Corollary \ref{cor-wp-thompson-a}, the word problem in  $T(G_2, \varphi)$ is decidable provided that it is decidable in $G_2$. By Corollary \ref{corollary-frattini-thompson-a}, the embedding $G \hookrightarrow T(G_2,\varphi)$ is Frattini. By Corollary \ref{corollary-isometry-thompson-a}, it is an isometric embedding provided that $G$ is finitely generated. Therefore, the embedding $G \hookrightarrow T(G_2,\varphi)$ satisfies Theorem \ref{T: thompson-1}.
\end{proof}	
	\appendix
	
\section{Thompson's embedding revisited}
\label{AS: Thompson's theorem}

 
Here we adapt the original embedding construction of  \cite{thompson_word_1980} to the setting of our paper and note that, in addition, it is an isometric embedding.

 \begin{thm}\label{T: thompson} Every computable group $G$ Frattini embeds into a finitely generated simple group $H$ with decidable word problem. Moreover, if $G$ is finitely generated, the embedding is isometric. 
\end{thm}

\begin{rem}
The original statement \cite{thompson_word_1980} is for finitely generated groups, but finite generation can be replaced by computability of $G$ due to Theorem \ref{L: embedding into 2 generated}. 
\end{rem}

\subsection{The embedding construction}

Let $G$ be a computable group. By Theorem \ref{T: splinter}, $G$ embeds into a finitely generated perfect group $G_1$ with decidable word problem (if $G$ is finitely generated, this claim also follows from \cite[\S 2]{thompson_word_1980}). 

Let $G_1=\{ g_1, g_2, \ldots \}$  be enumerated so that $\mathfrak{m}: \mathbb{N} \times \mathbb{N} \rightarrow \mathbb{N}$,  defined as $\mathfrak{m}((i, j)) = k$ if $g_i g_j =g_k$, is computable. By Remark \ref{R: computable group and word problem} the existence of  $\mathfrak{m}$ is equivalent to decidability of the word problem. 

Let $J=[ \frac{1}{2},1)$. For strictly positive $k\in \mathbb{N}$, let $$I_k:= \left[ \frac{2^k-1}{2^k}, \frac{2^k-1}{2^k}+\frac{1}{2^{2k}} \right).$$
We observe that any two such intervals are disjoint. 

We denote by $I^l_k$ the left half of the interval, and by $I^r_k$ the right half, so that 
\[
I^l_k=\left[ \frac{2^k-1}{2^k}, \frac{2^k-1}{2^k}+\frac{1}{2^{2k+1}} \right) \hbox{ and } I^r_k=\left[ \frac{2^k-1}{2^k}+\frac{1}{2^{2k+1}}, \frac{2^k-1}{2^k}+\frac{1}{2^{2k}} \right).
\]
 
For every $l\in \mathbb{N}$, let $\xi_l: J \rightarrow J$ be the piecewise homeomorphism, whose pieces are dyadic, and that, for every $k\in \mathbb{N}$, maps  $I_k^r$ onto $I_{\mathfrak{m}(l,k)}^r$ and that is the identity map elsewhere on $J$. 

Let us define $\lambda: G_1 \rightarrow \mathcal{C}(J)$  by $\lambda(g_l)=\xi_l$, for all $l\in \mathbb{N}$. 

\begin{rem}\label{R: thompson 1} The map $\lambda$ is an embedding of $G_1$ into computable maps in  $\mathcal{C}(J)$. 
\end{rem}
\begin{rem}\label{R: thompson 2}
If $\lambda(g_l)$ is fixing a right half $I^r_k$, then $g_l=1$. Indeed, then $\mathfrak{m}(l,k)=k$, so that $g_lg_k=g_k$. 
\end{rem}

\subsection{\texorpdfstring{$G_2$}{G}--dyadic maps}
Let $G_2 = \lambda(G_1)$. We study $G_2$--dyadic maps, see Definition \ref{G-dyadic}. 
 
Let $J'\subset J$ be an interval such that the closure of $J'$ contains 1. We want to prove:
 \begin{lem}\label{L: thompson} 
 {There is an algorithm to decide that a $G_2$--dyadic map $\Lambda: J' \to J'$ is equal to the identity.}
 \end{lem}

 \begin{rem} Our definition of $\lambda$ follows the construction in \cite[\S 3]{thompson_word_1980}. The main arguments to prove Lemma \ref{L: thompson}, cf.  Lemma \ref{L: thompson R1}, \ref{L: thompson R2}, \ref{L: thompson R3} and \ref{L: thompson R4} below, are essentially those of \cite[\S 3]{thompson_word_1980}. 
 \end{rem}
 
 \begin{rem} We assume that the closure of $I_n, \ldots ,I_0$, and $J_n,\ldots,J_0$ of Definition \ref{G-dyadic} contains $1$. This is no restriction in generality.
 
 Indeed, if $\Lambda:g_1f_1\cdots g_nf_n$ and $I_{i+1}$ (equivalently, $J_i$) does not contain $1$, then $g_{i}:J_i\to I_{i+1}$ has only finitely many dyadic pieces on $I_i$. Therefore, a finite sequence of chart subdivisions at the breakpoints of $g_i$  on $I_i$  transforms $f_{i+1}g_if_i$ into a finite number of dyadic maps. Thus, we can algorithmically split $\Lambda$ into a finite number of $G$-dyadic maps of $<n$ factors.1
 \end{rem}

\begin{lem}\label{L: thompson R0}  Let $\Lambda=g_nf_n\cdots g_1f_1:J'\to J'$ be a $G_2$--dyadic map. Then all dyadic factors $f_i$ of $\Lambda$ fix $1$.  
\end{lem}
  \begin{proof} By contradiction, let $f=f_{i}$ be a dyadic map such that  $f(1)\not=1$. Up to inverting $\Lambda$, $f(1)>1$. This contradicts the definition of $G_2$--dyadic maps, Definition \ref{G-dyadic}. 
 \end{proof}
 
 For $n\in \mathbb{Z}$, we write $s_n(x):= 2^{-n} x+(1-2^{-n}).$  All dyadic maps that fix $1$ are of this form. Note that $s_{n+m}=s_n \circ s_m$. 
We call $|n|$ the \emph{degree} of $s_n$.

\begin{lem}\label{L: thompson R1} Let $n\not = 0$,  and let $g\in G_2$. Then, for all  $k> |n|$, $gs_n$ and $s_n$ are equal on $I^r_{k}$. Moreover, $s_ng$ acts as the identity on at most finitely many $I_k^r$. 
\end{lem}

\begin{proof}
Let $n>0$ and $k>n$. 
Direct computations show that $$s_{-n}(I^r_{k}) = \left[\frac{2^{k-n}-1}{2^{k-n}}+\frac{1}{2^{2k-n+1}}, \frac{2^{k-n}-1}{2^{k-n}}+\frac{1}{2^{2k-n}}  \right) \subset I^l_{k-n}.$$
Since, by definition, $g$ acts trivially on $I^l_{k-n}$, we get  $gs_{-n}$ coincides with $s_{-n}$ on $I^r_{k}$.

 Similarly, $I^r_k\subset s_n(I^l_{k-n})$, so that $s_n(I_{k-n}^r)$ does not intersect with $I_l^r$, for any $l>0$. Thus, by definition, $g$ acts trivially on $s_n(I_{k-n}^r)$, and $gs_{n}$ coincides with $s_{n}$ on $I^r_{k-n}$. 
 
In addition, as $g$ permutes the intervals $I_k^r$, $s_ng$ acts as the identity on at most finitely many $I_k^r$.

\end{proof}

Let $m>0$ and for all $1\leqslant i \leqslant m$, let $g_i\not =1$  in $G_2$, and $n_i\not= 0 $ in $\mathbb{N}$.  
Let us fix $$\Lambda = g_ms_{n_m}\cdots g_1s_{n_1}$$ to be a $G_2$--dyadic map as in Lemma \ref{L: thompson}. 
Let $S_0=id $, $S_1=s_{n_1}$, and, recursively, $S_i=s_{n_i}S_{i-1}$. 

\begin{lem}\label{L: thompson R2} 
If, for all $i<m$, $S_i\not= id$ and $k$ is strictly larger than the degree of $S_i$, then  $\Lambda$ acts as $g_mS_{m}$ on  $I_k^r$.  
In particular, $\Lambda \not= id$. 
\end{lem}

 \begin{proof} Let $k$ be strictly larger than the degree of $S_i$, for all $i<m$. By Lemma \ref{L: thompson R1}, as $k>|n_1|$, $g_1s_{n_1}$ equals to $s_{n_1}$ on $I^r_{k}$. Thus, restricted to these intervals, $ g_ms_{n_m}\cdots g_1s_{n_2+n_1}$ equals to $\Lambda$. By induction this yields the first assertion. 
 
We show that  $g_m s_{n_m+\ldots +n_1} \not = id$ on all but finitely many of the intervals $I_k^r$. If $ s_{n_m+\ldots +n_1}\not=id$, this is by Lemma \ref{L: thompson R1}. Otherwise $g_m s_{n_m+\ldots +n_1} =g_m \not = id$, which yields the claim by Remark \ref{R: thompson 2}.
 \end{proof}

If $m>0$, let $i_0$ be the smallest index such that $n_{i_0}+\ldots +n_1=0$, and recursively define $i_j$ to be the smallest index such that $n_{i_j}+\ldots +n_{i_{j-1}+1}=0$. Let $i_M$ be the largest such index $<m$. 
\begin{lem}\label{L: thompson R3} If $n_m+ \ldots +n_2+n_1=0$, then $\Lambda$ equals to $g_mg_{i_{M}}g_{i_{M-1}}\cdots g_{i_1}g_{i_0}$ on all but a finite  number of  intervals $I_k^r$, which can be algorithmically determined. Otherwise, $\Lambda \not= id$. 
\end{lem}

\begin{proof}  If $m=0$ the claim follows by Lemma \ref{L: thompson R2}. Let $m>0$.
 
 By Lemma \ref{L: thompson R1}, $g_ms_{n_m}g_{m-1}\ldots g_{i_1}$ and $\Lambda$ are equal on $I^r_{k}$ unless $k$ is smaller than the degree of $S_i$, for some $i\leqslant i_0$. Inductively, $g_ms_{n_m}g_{m-1}\ldots g_{i_j}$ and $g_ms_{n_m}g_{m-1}\ldots g_{{i_{j-1}+1}}s_{i_{j-1}+1}$ equal on $I_{l_j}^r:=g_{i_{j-1}}\cdots g_{i_{0}}(I^r_{k})$ unless $l_j$ is smaller than the degree of $S_i$, for some $i_{j-1} <i\leqslant i_j$. Finally,  $g_ms_{n_m+\ldots + i_{M}+1}$ and $g_ms_{n_m}g_{m-1}\ldots g_{i_M+1}s_{M+1}$ are equal on $I_{l_M}=g_{i_{M}}\cdots g_{i_0}(I^r_{k})$, unless $l_M$ is smaller than the degree of $S_i$, for some $i_{M-1} <i\leqslant i_M$. 

Let $g:=g_{i_{M}}g_{i_{M-1}}\cdots g_{i_1}g_{i_0}$. 
We conclude that $\Lambda$ is equal to $g_ms_{m+\ldots +i_{M}+1}g$ on all but a finite number of intervals $I_k^r$. As the degree of the $S_i$ is computable, we can algorithmically determine these intervals. 
 If $s_{m+\ldots +i_{M}+1}=id$, this concludes the proof. Otherwise, by Lemma \ref{L: thompson R1}, $\Lambda$ acts as $s_{n_m+\ldots i_{M}+1}g$ on all but finitely many intervals $I_k^r$. Thus, $\Lambda\not=id$ by Lemma \ref{L: thompson R1}.    
 \end{proof}

 \begin{cor}\label{C: thompson A1} There is an algorithm to decide whether $\Lambda$ is the identity on the intervals $I_k^r$ in $J'$. \qed 
 \end{cor}
  \begin{proof}  By Lemma \ref{L: thompson R3}, there is a computable number $k_0>0$ such that, for all $k\geqslant k_0$, $\Lambda= id$ on $I_k^r$, if, and only if, $g_mg_{i_{M}}\cdots g_{i_1}g_{i_0}=1$. As the word problem in $G$ is decidable, this can be algorithmically determined.  
On the other hand, for each $k$, there is an (obvious) algorithm to decide whether or not $\Lambda$ acts as the identity on $I_k^r$. We apply this algorithm for each $k< k_0$. This completes the proof. 
 \end{proof}

 \begin{lem}\label{L: thompson R4} 
Let $x\in J'\setminus \bigcup_{k=1}^{\infty}\bigcup_{i=0}^m S_i^{-1}(I_k^r)$. Then $\Lambda(x)=S_m(x)$.
 \end{lem}
\begin{proof} Since, for all $k$,  $x\not \in S_{1}^{-1}(I_k^r)$, we have that $\Lambda(x)=g_ms_{m_n}\cdots g_{n_2}s_{n_2}s_{n_1}(x)$. By induction, $\Lambda(x)=S_m(x)$, which is the claim. 
\end{proof}

 \begin{proof}[Proof of Lemma \ref{L: thompson}] By Lemma \ref{L: thompson R0}, all dyadic factors in a $G_2$--dyadic map fix $1$. We first compute the degree of $S_m$. If the degree of $S_m$ is not $0$, then Lemma $\ref{L: thompson R3}$ implies that $\Lambda\not= id$. 
 
 Otherwise, Lemma \ref{L: thompson R4}  implies that $\Lambda$ is the identity on $ J'\setminus \bigcup _ {k=1}^{\infty}\bigcup_{i=0}^m S_i^{-1}(I_k^r)$. 
   Let $0\leqslant i \leqslant m$. We argue that there is an algorithm to decide whether or not $\Lambda$ is the identity on the intervals $S_i^{-1}(I_k^r)$ in $J'$. This will complete the proof. 
 
 Let $x\in S_i^{-1}(I_k^r)$, and let $y\in I_k^r$ be the point such that $x=S_i(y)$. We note that $\Lambda(x)=x$ if, and only if, $\Lambda(S_iy)=S_iy$, if, and only if, $S_i^{-1}\Lambda S_i (y)=y$.  Therefore, we need to decide whether or not the $G_2$--dyadic map $S_i^{-1}\Lambda S_i$ is the identity on the intervals $I_k^r$ such that $S_i^{-1}(I_k^r)\subset J'$.  
 
 Let $k_0>0$ be the smallest index such that for all $k\geqslant k_0$, $I_k^r \subset S_i(J')$. As $S_i(J')$ can be algorithmically determined, $k_0$ can be computed as well. Thus, we need to decide whether or not $S_i^{-1}\Lambda S_i$ is the identity on the intervals $I_k^r$ in $\left[\frac{2^{k_0}-1}{2^{k_0}},1\right)$. By Corollary \ref{C: thompson A1} such an algorithm exists.
 \end{proof}

\subsection{Frattini property}

To conclude Thompson's theorem, Theorem \ref{T: thompson}, we also need: 

\begin{lem} The embedding $\lambda: G_1 \to T(G_2,\varphi)$ is a Frattini embedding.
\end{lem} 

The proof of this lemma is analogous to the proof of Lemma \ref{L: Frattini left order}.  

\begin{proof}
Let $h,g \in G_2$ and $t\in T(G_2,\varphi)$. We assume that $ht^{-1}gt=1$. 

We represent $t$ by a canonical chart representation $(C_i\times I_i, C_j\times J_i, t_i)$ such that $\bigcup_{i}J_i=[0,1]$; and represent $t^{-1}$ by $(C_i\times J_i, C_i\times I_i, t_i)$. We recall that $t_i(I_i)=J_i$. 

Let $k$ be an index such that  $1$ is in the closure of $J_k$ and such that (after applying a chart refinement if necessary) $J_k\subseteq J$. As $g$ is fixing $1$, there is $J_k'\subseteq J_k$ such that $g(J_k')\subseteq  J_k$ and such that  $1/3$ is in the closure of $J_k'$. We let $I_k'=t_k^{-1}(J_k')$ and $I_k''= t_k^{-1}gt_k(I_k')$. 

Then the triple $(C\times I_k', C \times I_k'', t_k^{-1}gt_k)$ is in a canoncial chart representation of $t^{-1}gt$.  
 Up to applying the algorithm of Lemma \ref{L: composition 1} to this chart representation, we may assume that $I_k''$ is in $[0,1]$. Moreover, up to applying a chart refinement if necessary, we may assume that either $I_k''\cap \mathfrak{J}$ is empty or consists of one point $1$, or $I_k''\subseteq \mathfrak{J}$.

If $I_k''\cap J$ is empty or consists of one point, then $(C\times I_k', C \times I_k'', t_k^{-1}gt_k)$ is in a chart representation of $ht^{-1}gt$. Thus $t_k^{-1}gt_k=id$ on $I_k'$ by Lemma \ref{L: id}. This implies that $g$ acts as the identity on $J_k'$.  As $1$ is in the closure of $J_k'$, this implies that $g=1$.

Otherwise, $(C\times I_k', C \times h(I_k''), ht_k^{-1}gt_k)$ is in a chart representation of $ht^{-1}gt$.  Thus $ht_k^{-1}gt_k=id$ on $I_k'$ by Lemma \ref{L: id}. But then $t_kht_k^{-1}g:J_k'\to J_k'$ has to be the identity as well. As $g$ is fixing $1$, $t_kht_k^{-1}$ has to fix $1$. 

If $t_k$ does not fix $1$, $h$ acts (up to applying finitely many chart refinements if necessary) as a dyadic map on $I_k'$. But it has to fix $t_k^{-1}(1)$. Thus $h$ acts as the identity on $I_k'$. This implies that $g=1$ by Remark \ref{R: thompson 1}.

Otherwise, by Lemma \ref{L: thompson R3}, there are $g_{i_M}, \ldots, g_{i_0} \in G_2$ such that, on all but finitely many of the intervals $I_j^r$, the maps  $h^{-1}t_k^{-1}gt_k$ equals to $h^{-1}g_{i_0}^{-1}\cdots  g_{i_M}^{-1}gg_{i_M}\cdots g_{i_0}$. By Remark \ref{R: thompson 2}, this implies that $h$ and $g$ are conjugated in $G_1$. 
\end{proof}

Moreover, the embedding of Thompson is also an isometric embedding. 
	\begin{lem}\label{lem: isometry thompson embedding} The embedding $G_1 \hookrightarrow T(G_2,\varphi)$ is an isometric embedding.
\end{lem} 	
\begin{proof} We fix a finite generating set for $G_1$. This gives a finite generating set for $(G_2,\Phi)$. We denote by $|h|$ the word metric of $h$. 

Let $g\in G_2$ and $t\in T(G,\Phi)$ such that $tg=1$. We represent $t$ by finitely many (canonical) charts $(C_i\times I_i, C_i\times J_i, t_i)$ such that $\bigcup I_i=[0,1]$. We note that $|t_i|\leqslant |t|$. 

Let $I_k$ be the interval such that $1$ is in the closure of $I_k$ and such that (after applying a chart refinement if necessary) $I_k\subseteq J$. 

Then $(C_i\times g^{-1}(I_k), C_i\times J_i, t_ig)$ is in a canonical chart representation of $tg$. By Lemma \ref{L: id}, $t_ig$ is the identity mapping. In particular, $g^{-1}(I_k)=J_i$, so that $J_i\subseteq J$ and $1$ is in the closure of $J_i$. 

If $t_i$ is a dyadic map, then $t_i=id$ (Lemma \ref{L: thompson R1}) and thus $g=1$ . 

If $t_i \in G_2$,  then $g=t_i^{-1}$ and $|g|=|t_i|\leqslant |t|$.

Otherwise, $t_i=g_1f_1\cdots g_nf_n$ is a $G_2$--dyadic map. Moreover, by Remark \ref{R: thompson 2}, we may assume that all dyadic maps $f_i$ fix $1$. Thus, by Lemma \ref{L: thompson R3},  $g=g_{i_1}\cdots g_m$. Thus $|g|\leqslant m \leqslant |f|$.  

We conclude that the embedding is isometric. 
\end{proof}
 Combining Lemma \ref{lem: isometry thompson embedding} with Lemma \ref{lem-isometry-splinter gp}, we obtain
\begin{cor}
    If $G$ is finitely generated, then the embedding $G \hookrightarrow T(G_2, \varphi)$ is isometric. \qed
\end{cor}
\begin{proof}[Proof of Theorem \ref{T: thompson}] 
  By Lemmas \ref{L: thompson} and \ref{lemma-word problem in T-H-Phi}, the group $T(G_2,\varphi)$ has decidable word problem. This group is also  finitely generated and simple, see Lemmas \ref{Lem: fin gen of the main group} and  \ref{lem: the unifying lemma}. By construction, $G$ embeds into $T(G_2,\varphi)$.  
\end{proof}

\begin{rem} If $G$ is non-computably left-orderable with decidable word problem, it is open whether $T(G_2,\varphi)$ is left-orderable as well.
\end{rem}

\mbox{}

 \addtocontents{toc}{\setcounter{tocdepth}{-10}}
 
\bibliographystyle{alpha}

\bibliography{eilos}

\end{document}